\documentclass[11pt,a4paper]{article}

\usepackage[utf8]{inputenc}       
\usepackage[T1]{fontenc}          
\usepackage{microtype}            

\usepackage{amsmath, amssymb}     
\usepackage{amsthm}               
\usepackage{mathtools}            
\usepackage{bm}                   
\usepackage{comment}

\theoremstyle{plain}              
\newtheorem{theorem}{Theorem}[section]
\newtheorem{lemma}[theorem]{Lemma}
\newtheorem{corollary}[theorem]{Corollary}
\newtheorem{proposition}[theorem]{Proposition}

\theoremstyle{definition}         
\newtheorem{definition}[theorem]{Definition}

\usepackage{graphicx}             
\usepackage{float}                
\usepackage[font=small,labelfont=bf]{caption} 
\usepackage{subcaption}           

\usepackage{algorithm}            
\usepackage{algpseudocode}        

\usepackage[margin=2.5cm]{geometry}    
\usepackage{parskip}
\usepackage{xcolor} 
\usepackage[
colorlinks=true,
linkcolor=blue,
citecolor=green,
urlcolor=red
]{hyperref}
\usepackage[nameinlink]{cleveref} 

\crefname{equation}{equation}{equations}
\crefname{figure}{figure}{figures}
\crefname{table}{table}{tables}
\crefname{algorithm}{algorithm}{algorithms}

\title{\textbf{Element-wise Convergence Behavior of Subspace Iteration}}
\author{Mingyang Zhao\\
myzhao24@m.fudan.edu.cn
}
\date{\today}

\begin{document}

\maketitle


\begin{abstract}
This paper studies the element-wise convergence behavior of subspace iteration. All results are established for both \(\mathbb{F}=\mathbb{R}\) and \(\mathbb{F}=\mathbb{C}\). For a diagonal matrix \(\Lambda=\mathrm{diag}(\lambda_1,\ldots,\lambda_n)\) with \(|\lambda_1|>\cdots>|\lambda_n|>0\), we analyze the modulus of each entry of the iterative matrix sequence generated by subspace iteration, both without and with the Rayleigh--Ritz procedure. Under mild assumptions on the initial matrix \(X\in\mathbb{F}^{n\times m}\), we first derive exact asymptotic expressions for \(|Q_k(i,j)|\) in the subspace iteration without the Rayleigh--Ritz procedure: entries with \(i\neq j\) decay as \(|\lambda_{\max\{i,j\}}/\lambda_{\min\{i,j\}}|^k\), and the deviation of \(|Q_k(j,j)|\) from \(1\) decays as \(\max\{|\lambda_j/\lambda_{j-1}|,|\lambda_{j+1}/\lambda_j|\}^{2k}\), with explicit coefficients determined by the LU factorization of \(X\). Similar conclusions could also be established for upper triangular matrices. For the Rayleigh--Ritz variant, we obtain element-wise asymptotic formulas for the normalized Ritz vectors \(Z_k\). Specifically, for \(|Z_k(i,j)|\), off-diagonal entries with \(i\geq m+1\) decay as \(|\lambda_i/\lambda_j|^k\), off-diagonal entries with \(i\leq m\) decay as \(|\lambda_{m+1}^2/(\lambda_i\lambda_j)|^k\), and the deviation of \(|Z_k(j,j)|\) from \(1\) decays as \(|\lambda_{m+1}/\lambda_j|^{2k}\). These results give an element-wise description of the convergence behavior of subspace iteration. After an orthogonal or unitary change of basis, the results could be further generalized.
\end{abstract}


\section{Introduction}

Subspace iteration is a fundamental iterative method for computing a dominant invariant subspace of a large square matrix. It can be viewed as a block generalization of the power method: instead of iterating a single vector, one propagates an \(n\times m\) block and orthonormalizes its columns at each step. Because of its simple structure and its close relation to the QR algorithm and Krylov subspace methods, it remains a standard tool for large-scale eigenvalue problems and has been widely used in scientific computing and data analysis; see, e.g., \cite{golub_book,parlett_book,saad_book}.

The pseudocode of subspace iteration without the Rayleigh--Ritz procedure is as follows \cite{saad_book}.

\begin{algorithm}[H]
\caption{Subspace iteration without the Rayleigh--Ritz procedure}
\begin{algorithmic}[1]
\Require \(A\in\mathbb{C}^{n\times n}\), \(X\in\mathbb{C}^{n\times m}\)
\State \(Q_0R_0=X\)
\For{\(k=1,2,\dots\)}
\State \(Q_kR_k=AQ_{k-1}\)
\EndFor
\end{algorithmic}
\end{algorithm}

For simplicity, assume that \(A=\mathrm{diag}(\lambda_1,\dots,\lambda_n)\) with \(|\lambda_1|>\dots>|\lambda_n|>0\). Under mild assumptions on \(X\), we have \(|Q_k|\to\begin{bmatrix}
I_m\\
O
\end{bmatrix}\), that is,
\[
|Q_k(i,j)|\to\delta_{i,j},\quad (i,j)\in[n]\times[m].
\]
However, what is the convergence behavior of the modulus of each entry \(|Q_k(i,j)|\)?

The pseudocode of subspace iteration with the Rayleigh--Ritz procedure is as follows \cite{saad_book}.

\begin{algorithm}[H]
\caption{Subspace iteration with the Rayleigh--Ritz procedure}
\begin{algorithmic}[1]
\Require \(A\in\mathbb{C}^{n\times n}\), \(X\in\mathbb{C}^{n\times m}\)
\State \(Q_0R_0=X\)
\State \(Q_0^*AQ_0=U_0\Theta_0U_0^{-1}\) (or Schur decomposition \(Q_0^*AQ_0=U_0\Theta_0U_0^*\))
\State \(Z_0=Q_0U_0\)
\For{\(k=1,2,\dots\)}
\State \(Q_kR_k=AZ_{k-1}\)
\State \(Q_k^*AQ_k=U_k\Theta_kU_k^{-1}\) (or Schur decomposition \(Q_k^*AQ_k=U_k\Theta_kU_k^*\))
\State \(Z_k=Q_kU_k\)
\EndFor
\end{algorithmic}
\end{algorithm}

For simplicity, assume that \(A=\mathrm{diag}(\lambda_1,\dots,\lambda_n)\) with \(|\lambda_1|>\dots>|\lambda_n|>0\) Under mild assumptions on \(X\), we have \(|Z_k|\to\begin{bmatrix}
I_m\\
O
\end{bmatrix}\), that is,
\[
|Z_k(i,j)|\to\delta_{i,j},\quad (i,j)\in[n]\times[m].
\]
However, what is the convergence behavior of the modulus of each entry \(|Z_k(i,j)|\)?

Classical convergence theory for subspace iteration is usually formulated either at the subspace level or at the level of individual eigenvectors and Ritz vectors. It is worth emphasizing that these classical results are not element-wise: they do not describe the decay or the deviation of individual entries of the computed basis matrices \(Q_k\) or \(Z_k\).

At the subspace level, Theorem~7.3.1 in \cite{golub_book} analyzes orthogonal iteration, which is equivalent to subspace iteration in the absence of the Rayleigh--Ritz procedure. It shows that the distance between the computed subspace \(\operatorname{ran}(Q_k)\) and the dominant invariant subspace \(D_r(A)\) is bounded by a constant times \(|\lambda_{r+1}/\lambda_r|^k\). The constant depends on the separation of the dominant eigenvalues and on the departure from normality of \(A\). This is a subspace-level convergence result: it measures the convergence of the whole computed subspace, not the convergence of its individual entries.

At the vector level, Theorem~14.4.1 in \cite{parlett_book} gives a convergence estimate for each target eigenvector \(z_i\) in the symmetric case. More precisely, it shows that a certain vector \(x_i^{(k)}\) constructed from the iterated subspace approximates \(z_i\) linearly, with a rate governed by ratios of the relevant eigenvalues. Theorem~14.4.2 in \cite{parlett_book} further shows that the Ritz vector \(y_i^{(k)}\) is asymptotically equivalent to \(x_i^{(k)}\), and therefore the Ritz vector itself converges to the corresponding eigenvector at the same vector-level rate. These are vector-level results for individual eigenvectors and Ritz vectors.

Also at the vector level, Theorem~5.2 in \cite{saad_book} considers subspace iteration with projection. For each wanted eigenvector \(u_i\), it estimates the distance from \(u_i\) to the iterated subspace \(S_k\). The bound decays like \(|\lambda_{m+1}/\lambda_i|^k\), up to a factor that tends to zero as \(k\to\infty\). This is again a per-eigenvector estimate: it describes how fast each wanted eigenvector is captured by the iterated subspace, rather than how the subspace as a whole converges or how individual entries of the basis behave.

In summary, the existing theory provides either subspace-level bounds, as in Theorem~7.3.1 of \cite{golub_book}, or vector-level bounds, as in Theorems~14.4.1 and~14.4.2 of \cite{parlett_book} and Theorem~5.2 of \cite{saad_book}. None of these results gives an element-wise description of the convergence of subspace iteration. The purpose of this paper is to fill this gap by analyzing the modulus of each entry of the iterative matrices, both without and with the Rayleigh--Ritz procedure.

The main contributions of this paper are as follows.
\begin{itemize}
\item For subspace iteration without the Rayleigh--Ritz procedure, we derive exact asymptotic expressions for the modulus of every entry of the orthonormal basis \(Q_k\). For \(i\ne j\), the entry \(|Q_k(i,j)|\) decays as
\[
\left|\frac{\lambda_{\max\{i,j\}}}{\lambda_{\min\{i,j\}}}\right|^k,
\]
while the deviation of \(|Q_k(j,j)|\) from \(1\) decays as
\[
\max\left\{\left|\frac{\lambda_j}{\lambda_{j-1}}\right|,
\left|\frac{\lambda_{j+1}}{\lambda_j}\right|\right\}^{2k}.
\]
The coefficients in these asymptotic expressions are explicitly determined by the LU factorization of the initial matrix \(X\).
\item For subspace iteration with the Rayleigh--Ritz procedure, we derive exact asymptotic expressions for the modulus of every entry of the normalized Ritz vectors \(Z_k\). In particular, off-diagonal entries with \(i\ge m+1\) decay as
\[
\left|\frac{\lambda_i}{\lambda_j}\right|^k,
\]
off-diagonal entries with \(i\le m\) decay as
\[
\left|\frac{\lambda_{m+1}^2}{\lambda_i\lambda_j}\right|^k,
\]
and the deviation of \(|Z_k(j,j)|\) from \(1\) decays as
\[
\left|\frac{\lambda_{m+1}}{\lambda_j}\right|^{2k}.
\]
\item All results are proved for both \(\mathbb{F}=\mathbb{R}\) and \(\mathbb{F}=\mathbb{C}\). After an orthogonal or unitary change of basis, the conclusions could be further generalized.
\end{itemize}

Although the entrywise convergence of subspace iteration has also been studied in \cite{NEURIPS2020_3d8e03e8}, our work is fundamentally different. The analysis in \cite{NEURIPS2020_3d8e03e8} measures the error in the \(\ell_{2\to\infty}\) norm and provides only upper bounds, whereas we establish element-wise asymptotic results: exact expressions with explicit coefficients for the case without the Rayleigh-Ritz procedure, and element-wise upper and lower bounds for the Rayleigh-Ritz variant. This difference in both the convergence metric and the precision of the results makes the two analyses essentially distinct.

The rest of this paper is organized as follows. Section~\ref{section:withoutRR diagonal} is devoted to subspace iteration without the Rayleigh--Ritz procedure. We first recall the LU factorization of the initial block and the associated scaled lower triangular matrices, and then prove the exact asymptotic element-wise formulas by induction. Section \ref{section:withoutRR upper triangular} generalizes the conclusion of section \ref{section:withoutRR diagonal} to upper triangular matrices. Section~\ref{section:withRR diagonal eigen} treats the (eigen) Rayleigh--Ritz variant. After introducing the necessary notation and auxiliary lemmas, we construct a complete metric space and a contraction map, apply the Banach fixed point theorem, and finally translate the fixed point estimates back to the normalized Ritz vectors. Section~\ref{section:withRR diagonal Schur} treats the (Schur) Rayleigh--Ritz variant. Building on the proof in Section~\ref{section:withRR diagonal eigen}, we conduct further analysis, and the final result has exactly the same form as in Section~\ref{section:withRR diagonal eigen}. Section~\ref{section:conclusion} concludes the paper.

\section{Subspace iteration without the Rayleigh--Ritz procedure I}\label{section:withoutRR diagonal}

In this section, we derive the element-wise convergence behavior of subspace iteration without the Rayleigh--Ritz procedure, where the matrix \(\Lambda\) for left multiplication to be diagonal.

\begin{lemma}\label{lemma:uniqueness of LU}
Let \(\mathbb{F}\) be \(\mathbb{R}\) or \(\mathbb{C}\), and let \(n\), \(m\in\mathbb{N}_+\) with \(m\leq n\). Let \(X\in\mathbb{F}^{n\times m}\) satisfy
\[
\det(X(1:j,1:j))\neq 0,\qquad\forall j\in[m].
\]
Then there exists a unique unit lower triangular \(L\in\mathbb{F}^{n\times m}\) and a unique upper triangular \(U\in\mathbb{F}^{m\times m}\) such that \(X=LU\).
\end{lemma}
\begin{proof}
See appendix \ref{proof:uniqueness of LU}.
\end{proof}

\begin{theorem}\label{theorem:withoutRR diagonal}
Suppose \(\mathbb{F}=\mathbb{R}\) or \(\mathbb{C}\), \(n\), \(m\in\mathbb{N}_+\) with \(m\leq n\), and \(\Lambda=\mathrm{diag}(\lambda_1,\dots,\lambda_n)\in\mathbb{F}^{n\times n}\), with \(|\lambda_1|>\dots>|\lambda_n|>0\). Assume \(X\in\mathbb{F}^{n\times m}\) satisfies \(\det(X(1:j,1:j))\neq 0\) for every \(j\in[m]\). Then there exists a unique unit lower triangular \(L\in\mathbb{F}^{n\times m}\) and upper triangular \(U\in\mathbb{F}^{m\times m}\) such that \(X=LU\). Assume further that all strictly lower triangular entries of \(L\) and of \(L(1:m,:)^{-1}\) are nonzero.

Let \(Q_kR_k=\Lambda^k X\) be a QR factorization of \(\Lambda^kX\). Then for every \(j\in[m]\), we have
\[
|Q_k(i,j)|=\begin{cases}
|(L(1:m,:)^{-1})(j,i)|\left|\dfrac{\lambda_j}{\lambda_i}\right|^k(1+o(1)),\quad&i\in[j-1]\\
1+o(1),\quad&i=j\\
|L(i,j)|\left|\dfrac{\lambda_i}{\lambda_j}\right|^k(1+o(1)),\quad&i\in[n]\setminus[j]
\end{cases}.
\]
\end{theorem}
\begin{proof}
See the rest of this section.
\end{proof}

\begin{corollary}
Suppose \(\mathbb{F}=\mathbb{R}\) or \(\mathbb{C}\), \(n\), \(m\in\mathbb{N}_+\) with \(m\leq n\) and \(\Lambda=\mathrm{diag}(\lambda_1,\dots,\lambda_n)\in\mathbb{F}^{n\times n}\), with \(|\lambda_1|>\dots>|\lambda_n|>0\). Suppose \(V\in\mathbb{F}^{n\times n}\) satisfies \(VV^*=V^*V=I_n\) and define \(A:=V\Lambda V^*\). Suppose \(Y\in\mathbb{F}^{n\times m}\) and define \(X:=V^*Y\). Assume \(\det(X(1:j,1:j))\neq 0\) for every \(j\in[m]\). Then there exists a unique unit lower triangular \(L\in\mathbb{F}^{n\times m}\) and upper triangular \(U\in\mathbb{F}^{m\times m}\) such that \(X=LU\). Suppose further that all strictly lower triangular entries of \(L\) and of \(L(1:m,:)^{-1}\) are nonzero.

Let \(\tilde{Q}_k\tilde{R}_k=A^kY\) be a QR factorization of \(A^kY\). Then for every \(j\in[m]\), we have
\[
|V(:,i)^*\tilde{Q}_k(:,j)|=\begin{cases}
|(L(1:m,:)^{-1})(j,i)|\left|\dfrac{\lambda_j}{\lambda_i}\right|^k(1+o(1)),\quad&i\in[j-1]\\
1+o(1),\quad&i=j\\
|L(i,j)|\left|\dfrac{\lambda_i}{\lambda_j}\right|^k(1+o(1)),\quad&i\in[n]\setminus[j]
\end{cases}.
\]

In particular, we have
\[
|V(:,i)^*\tilde{Q}_k(:,j)|=\Theta\left(\left|\dfrac{\lambda_{\max\{i,j\}}}{\lambda_{\min\{i,j\}}}\right|^k\right),\quad i\neq j;
\]
\[
1-|V(:,j)^*\tilde{Q}_k(:,j)|=\Theta\left(\max\left\{\left|\dfrac{\lambda_j}{\lambda_{j-1}}\right|,\left|\dfrac{\lambda_{j+1}}{\lambda_j}\right|\right\}^{2k}\right).
\]
\end{corollary}
\begin{proof}
This follows by applying a change of orthonormal basis to the conclusion of Theorem \ref{theorem:withoutRR diagonal}.
\end{proof}

We now prove Theorem~\ref{theorem:withoutRR diagonal}.

\subsection{LU factorization of \(X\)}

Since \(\det(X(1:j,1:j))\neq 0\) for all \(j\in[m]\), by Lemma \ref{lemma:uniqueness of LU}, there exists a unique unit lower triangular \(L\in\mathbb{F}^{n\times m}\) and upper triangular \(U\in\mathbb{F}^{m\times m}\) such that \(X=LU\). Let \(\Lambda_m=\Lambda(1:m,1:m)\). Set \(L_k=\Lambda^kL\Lambda_m^{-k}\). Then \(L_k\in\mathbb{F}^{n\times m}\) remains unit lower triangular, and its strictly lower triangular entries are:
\[
L_k(i,j)=L(i,j)\left(\dfrac{\lambda_i}{\lambda_j}\right)^k,\quad\forall i>j.
\]

\subsection{Gram--Schmidt orthogonalization of the columns of \(L_k\)}

For \(j=1,2,\dots,m\), define \(V_k(:,j)\) and \(Q_k(:,j)\) inductively by
\[
V_k(:,j)=L_k(:,j)-\sum_{l=1}^{j-1}\langle L_k(:,j),Q_k(:,l)\rangle Q_k(:,l),\quad Q_k(:,j)=\dfrac{V_k(:,j)}{\|V_k(:,j)\|_2}.
\]

Claim: For every \(j\in[m]\), we have
\[
Q_k(i,j)=\begin{cases}
\overline{(L(1:m,:)^{-1})(j,i)}\overline{\left(\dfrac{\lambda_j}{\lambda_i}\right)^k}(1+o(1)),\quad&i\in[j-1]\\
1+o(1),\quad&i=j\\
L(i,j)\left(\dfrac{\lambda_i}{\lambda_j}\right)^k(1+o(1)),\quad&i\in[n]\setminus[j]
\end{cases}.
\]

We prove the claim by induction.

\subsection{Base case}

Since \(\|V_k(:,1)\|_2=\|L_k(:,1)\|_2=1+o(1)\), for every \(i\in[n]\) we have
\[
Q_k(i,1)=V_k(i,1)/\|V_k(:,1)\|_2=L_k(i,1)(1+o(1))=L(i,1)\left(\dfrac{\lambda_i}{\lambda_1}\right)^k(1+o(1)).
\]

\subsection{Inductive step}

Assume the claim holds for \(1\), \(2\), \(\dots\), \(j-1\). We prove it for \(j\).

For \(l\in[j-1]\), define \(c_{l,j}=\langle L_k(:,j), Q_k(:,l)\rangle\). Then
\begin{align*}
c_{l,j}&=\sum_{i=j}^{n}L_k(i,j)\overline{Q_k(i,l)}\\
&=\sum_{i=j}^{n}L(i,j)\left(\dfrac{\lambda_i}{\lambda_j}\right)^k\overline{L(i,l)\left(\dfrac{\lambda_i}{\lambda_l}\right)^k}(1+o(1))\\
&=\overline{L(j,l)\left(\dfrac{\lambda_j}{\lambda_l}\right)^k}(1+o(1)).
\end{align*}

\subsubsection{Computation of \(V_k(i,j)\) (\(i>j\))}

\begin{align*}
V_k(i,j)&=L_k(i,j)-\sum_{l=1}^{j-1}c_{l,j}Q_k(i,l)\\
&=L(i,j)\left(\dfrac{\lambda_i}{\lambda_j}\right)^k-\sum_{l=1}^{j-1}\overline{L(j,l)\left(\dfrac{\lambda_j}{\lambda_l}\right)^k}(1+o(1))L(i,l)\left(\dfrac{\lambda_i}{\lambda_l}\right)^k(1+o(1))\\
&=L(i,j)\left(\dfrac{\lambda_i}{\lambda_j}\right)^k(1+o(1)).
\end{align*}

\subsubsection{Computation of \(V_k(j,j)\)}

\begin{align*}
V_k(j,j)&=L_k(j,j)-\sum_{l=1}^{j-1}c_{l,j}Q_k(j,l)\\
&=1-\sum_{l=1}^{j-1}\overline{L(j,l)\left(\dfrac{\lambda_j}{\lambda_l}\right)^k}(1+o(1))L(j,l)\left(\dfrac{\lambda_j}{\lambda_l}\right)^k(1+o(1))\\
&=1+o(1).
\end{align*}

\subsubsection{Computation of \(V_k(i,j)\) (\(i<j\))}

\begin{align*}
V_k(i,j)&=L_k(i,j)-\sum_{l=1}^{j-1}c_{l,j}Q_k(i,l)\\
&=-\sum_{l=1}^{i-1}c_{l,j}Q_k(i,l)-\sum_{l=i}c_{l,j}Q_k(i,l)-\sum_{l=i+1}^{j-1}c_{l,j}Q_k(i,l)\\
&=-\sum_{l=1}^{i-1}\overline{L(j,l)\left(\dfrac{\lambda_j}{\lambda_l}\right)^k}(1+o(1))L(i,l)\left(\dfrac{\lambda_i}{\lambda_l}\right)^k(1+o(1))\\
&\quad-\overline{L(j,i)\left(\dfrac{\lambda_j}{\lambda_i}\right)^k}(1+o(1))(1+o(1))\\
&\quad-\sum_{l=i+1}^{j-1}\overline{L(j,l)\left(\dfrac{\lambda_j}{\lambda_l}\right)^k}(1+o(1))\overline{(L(1:m,:)^{-1})(l,i)\left(\dfrac{\lambda_l}{\lambda_i}\right)^k}(1+o(1))\\
&=\left(-\overline{L(j,i)}-\sum_{l=i+1}^{j-1}\overline{L(j,l)(L(1:m,:)^{-1})(l,i)}\right)\overline{\left(\dfrac{\lambda_j}{\lambda_i}\right)^k}(1+o(1)).
\end{align*}

Since \(L(1:m,:)\) and \(L(1:m,:)^{-1}\) are both unit lower triangular matrices, we have
\begin{align*}
0&=(L(1:m,:)L(1:m,:)^{-1})(j,i)=\sum_{l=i}^{j}L(j,l)L(1:m,:)^{-1}(l,i)\\
&=L(j,i)+\sum_{l=i+1}^{j-1}L(j,l)L(1:m,:)^{-1}(l,i)+L(1:m,:)^{-1}(j,i),
\end{align*}
so
\[
\overline{L(1:m,:)^{-1}(j,i)}=-\overline{L(j,i)}-\sum_{l=i+1}^{j-1}\overline{L(j,l)L(1:m,:)^{-1}(l,i)}.
\]
Hence
\[
V_k(i,j)=\overline{L(1:m,:)^{-1}(j,i)}\overline{\left(\dfrac{\lambda_j}{\lambda_i}\right)^k}(1+o(1)).
\]

\subsubsection{Normalization in the 2-norm}

Since \(\|V_k(:,j)\|_2=1+o(1)\), we have \(Q_k(:,j)=V_k(:,j)(1+o(1))\), hence the claim holds for \(j\).

\subsection{Completing the proof of Theorem \ref{theorem:withoutRR diagonal}}

By the base case and the inductive step, the claim holds.

Since
\[
\Lambda^kX=\Lambda^kL\Lambda_m^{-k}\Lambda_m^kU=Q_k(\hat{R}_k\Lambda_m^kU)
\]
for some upper triangular \(\hat{R}_k\), the \(Q\)-factor in the QR factorization of \(\Lambda^kX\) is \(Q_k\). The claim gives
\[
|Q_k(i,j)|=\begin{cases}
|(L(1:m,:)^{-1})(j,i)|\left|\dfrac{\lambda_j}{\lambda_i}\right|^k(1+o(1)),\quad&i\in[j-1]\\
1+o(1),\quad&i=j\\
|L(i,j)|\left|\dfrac{\lambda_i}{\lambda_j}\right|^k(1+o(1)),\quad&i\in[n]\setminus[j]
\end{cases}.
\]
By the uniqueness of the QR factorization, different choices of the QR factorization do not affect \(|Q_k(i,j)|\), which completes the proof of Theorem~\ref{theorem:withoutRR diagonal}.

\section{Subspace iteration without the Rayleigh--Ritz procedure II}\label{section:withoutRR upper triangular}

In this section, we derive the element-wise convergence behavior of subspace iteration without the Rayleigh--Ritz procedure, where the matrix \(T\) for left multiplication to be upper triangular.

\begin{lemma}\label{lemma:eigen decomposition of T}
Let \(\mathbb{F}\) be \(\mathbb{R}\) or \(\mathbb{C}\), and let \(T\in\mathbb{F}^{n\times n}\) be upper triangular with distinct diagonal entries. Put \(\lambda_i:=T(i,i)\) for \(i\in[n]\) and \(\Lambda:=\mathrm{diag}(\lambda_1,\dots,\lambda_n)\).

Then there exists a unique unit upper triangular \(W\in\mathbb{F}^{n\times n}\) such that \(T=W\Lambda W^{-1}\).
\end{lemma}
\begin{proof}
See appendix \ref{proof:eigen decomposition of T}.
\end{proof}

\begin{definition}
Let \(\mathbb{F}\) be \(\mathbb{R}\) or \(\mathbb{C}\), \(n\), \(m\in\mathbb{N}_+\) with \(m\leq n\). For \(X\in\mathbb{F}^{n\times m}\) of full column rank, use \(Q(X)\) denote the \(Q\)-factor of a compact QR factorization of \(X\). By the uniqueness of QR factorization, \(|Q(X)|\) is well defined.
\end{definition}

\begin{theorem}\label{theorem:withoutRR upper triangular}
Suppose \(\mathbb{F}=\mathbb{R}\) or \(\mathbb{C}\), \(n\), \(m\in\mathbb{N}_+\) with \(m\leq n\), and \(T\in\mathbb{F}^{n\times n}\) upper triangular with \(|T(1,1)|>|T(2,2)|>\dots>|T(n,n)|>0\). Put \(\lambda_i:=T(i,i)\) for \(i\in[n]\) and define \(\Lambda:=\mathrm{diag}(\lambda_1,\dots,\lambda_n)\). Then there exists a unique unit upper triangular \(W\in\mathbb{F}^{n\times n}\) such that \(T=W\Lambda W^{-1}\). Assume \(X\in\mathbb{F}^{n\times m}\) with \(\det((W^{-1}X)(1:j,1:j))\neq 0\) for all \(j\in[m]\). Then there exists a unique unit lower triangular \(L\in\mathbb{F}^{n\times m}\) and a unique upper triangular \(U\in\mathbb{F}^{m\times m}\) such that \(W^{-1}X=LU\). Assume further that all strictly lower triangular entries of \(L\) and of \(L(1:m,:)^{-1}\) are nonzero.

Then \(\forall j\in[m]\),
\[
|Q(T^kX)(i,j)|=\begin{cases}
|(L(1:m,:)^{-1})(j,i)|\left|\dfrac{\lambda_j}{\lambda_i}\right|^k(1+o(1)),&i\in[j-1]\\
1+o(1),&i=j\\
|L(i,j)|\left|\dfrac{\lambda_i}{\lambda_j}\right|^k(1+o(1)),&i\in[n]\setminus[j]
\end{cases}.
\]
\end{theorem}
\begin{proof}
See subsection \ref{proof:withoutRR upper triangular}.
\end{proof}

\begin{corollary}
Suppose \(\mathbb{F}=\mathbb{R}\) or \(\mathbb{C}\), \(n\), \(m\in\mathbb{N}_+\) with \(m\leq n\), and \(T\in\mathbb{F}^{n\times n}\) upper triangular with \(|T(1,1)|>|T(2,2)|>\dots>|T(n,n)|>0\). Put \(\lambda_i:=T(i,i)\) for \(i\in[n]\) and define \(\Lambda:=\mathrm{diag}(\lambda_1,\dots,\lambda_n)\). Then there exists a unique unit upper triangular \(W\in\mathbb{F}^{n\times n}\) such that \(T=W\Lambda W^{-1}\). Suppose \(V\in\mathbb{F}^{n\times n}\) satisfies \(VV^*=V^*V=I_n\) and define \(A:=VTV^*\). Suppose \(Y\in\mathbb{F}^{n\times m}\) and define \(X:=V^*Y\). Assume \(\det((W^{-1}X)(1:j,1:j))\neq 0\) for all \(j\in[m]\). Then there exists a unique unit lower triangular \(L\in\mathbb{F}^{n\times m}\) and a unique upper triangular \(U\in\mathbb{F}^{m\times m}\) such that \(W^{-1}X=LU\). Assume further that all strictly lower triangular entries of \(L\) and of \(L(1:m,:)^{-1}\) are nonzero.

Then \(\forall j\in[m]\),
\[
|V(:,i)^*Q(A^kY)(:,j)|=\begin{cases}
|(L(1:m,:)^{-1})(j,i)|\left|\dfrac{\lambda_j}{\lambda_i}\right|^k(1+o(1)),&i\in[j-1]\\
1+o(1),&i=j\\
|L(i,j)|\left|\dfrac{\lambda_i}{\lambda_j}\right|^k(1+o(1)),&i\in[n]\setminus[j]
\end{cases}.
\]
\end{corollary}
\begin{proof}
This follows by applying a change of orthonormal basis to the conclusion of Theorem \ref{theorem:withoutRR upper triangular}.
\end{proof}

\subsection{Proof of Theorem \ref{theorem:withoutRR upper triangular}}\label{proof:withoutRR upper triangular}

For \(j\in[n]\), define \(\Lambda_j=\Lambda(1:j,1:j)\), \(\Lambda_j^{\perp}:=\Lambda(j+1:n,j+1:n)\). Define \(L_k=\Lambda^kL\Lambda_m^{-k}\in\mathbb{F}^{n\times m}\).

Given that
\[
T^kX=W\Lambda^kW^{-1}X=W\Lambda^kLU=W\Lambda^kL\Lambda_m^{-k}\Lambda_m^kU=WL_k\Lambda_m^kU,
\]
we have \(Q(T^kX)=Q(WL_k)\).

Define \(\rho:=\max_{i\in[n-1]}\left|\dfrac{\lambda_{i+1}}{\lambda_i}\right|\in(0,1)\).

\begin{lemma}\label{lemma:key lemma}
There exists \(k_0=k_0(n,m,T,X)\in\mathbb{N}_+\) and \(C_0=C_0(n,m,T,X)\in[0,+\infty)\) such that, if \(k\geq k_0\), then \(\forall(i,j)\in[n]\times[m]\),
\[
||Q(WL_k)(i,j)|-|Q(L_k)(i,j)||\leq C_0\rho^k\left|\dfrac{\lambda_{\max\{i,j\}}}{\lambda_{\min\{i,j\}}}\right|^k.
\]
\end{lemma}
\begin{proof}
See appendix \ref{proof:key lemma}.
\end{proof}

We have proved in Theorem \ref{theorem:withoutRR diagonal} that \(\forall j\in[m]\),
\[
|Q(L_k)(i,j)|=\begin{cases}
|(L(1:m,:)^{-1})(j,i)|\left|\dfrac{\lambda_j}{\lambda_i}\right|^k(1+o(1)),&i\in[j-1]\\
1+o(1),&i=j\\
|L(i,j)|\left|\dfrac{\lambda_i}{\lambda_j}\right|^k(1+o(1)),&i\in[n]\setminus[j]
\end{cases}.
\]
Using Lemma \ref{lemma:key lemma}, we know that \(\forall j\in[m]\),
\[
|Q(T^kX)(i,j)|=|Q(WL_k)(i,j)|=\begin{cases}
|(L(1:m,:)^{-1})(j,i)|\left|\dfrac{\lambda_j}{\lambda_i}\right|^k(1+o(1)),&i\in[j-1]\\
1+o(1),&i=j\\
|L(i,j)|\left|\dfrac{\lambda_i}{\lambda_j}\right|^k(1+o(1)),&i\in[n]\setminus[j]
\end{cases}.
\]
The proof of Theorem \ref{theorem:withoutRR upper triangular} is done.

\section{Subspace iteration with the Rayleigh--Ritz procedure I}\label{section:withRR diagonal eigen}

In this section, we derive the element-wise convergence behavior of subspace iteration with the (eigen) Rayleigh--Ritz procedure, where the matrix \(\Lambda\) for left multiplication to be diagonal.

\begin{theorem}\label{theorem:withRR diagonal eigen}
Suppose \(\mathbb{F}=\mathbb{R}\) or \(\mathbb{C}\), \(n\), \(m\in\mathbb{N}_+\) with \(n>m\), and \(\Lambda=\mathrm{diag}(\lambda_1,\dots,\lambda_n)\in\mathbb{F}^{n\times n}\), with \(|\lambda_1|>\dots>|\lambda_n|>0\). Assume \(X\in\mathbb{F}^{n\times m}\) is such that \(X(1:m,:)\) is invertible and all entries of \(S:=X(m+1:n,:)X(1:m,:)^{-1}\in\mathbb{F}^{(n-m)\times m}\) are nonzero. 

Then there exists \(k_4=k_4(n,m,\Lambda,X)\in\mathbb{N}_+\) such that \(\forall k\geq k_4\),

\begin{itemize}
\item The \(m\) Ritz values of \(\Lambda\) projected on \(\mathrm{span}(\Lambda^kX)\) have distinct modulus; denote them by \(\{\theta_j\}_{j\in[m]}\) where \(|\theta_1|>\dots>|\theta_m|\).
\item Let \(Z_k^e(:,j)\) be a unit (eigen) Ritz vector corresponding to \(\theta_j\). Then \(\forall j\in[m]\), we have
\begin{align*}
|Z_k^e(i,j)|&=\left|\dfrac{\lambda_j-\lambda_{m+1}}{\lambda_i-\lambda_j}\right||S(1,i)||S(1,j)|\left|\dfrac{\lambda_{m+1}^2}{\lambda_i\lambda_j}\right|^k(1+o(1)),\quad i\in[m]\setminus\{j\},\\
|Z_k^e(i,j)|&=|S(i-m,j)|\left|\dfrac{\lambda_i}{\lambda_j}\right|^k(1+o(1)),\quad i\in[n]\setminus[m],\\
1-|Z_k^e(j,j)|&=\dfrac{1}{2}|S(1,j)|^2\left|\dfrac{\lambda_{m+1}}{\lambda_j}\right|^{2k}(1+o(1)).
\end{align*}
\end{itemize}
\end{theorem}
\begin{proof}
See the rest of this section.
\end{proof}

\begin{corollary}
Suppose \(\mathbb{F}=\mathbb{R}\) or \(\mathbb{C}\), \(n\), \(m\in\mathbb{N}_+\) with \(n>m\), and \(\Lambda=\mathrm{diag}(\lambda_1,\dots,\lambda_n)\in\mathbb{F}^{n\times n}\), with \(|\lambda_1|>\dots>|\lambda_n|>0\). Assume \(V\in\mathbb{F}^{n\times n}\) satisfies \(VV^*=V^*V=I_n\). Define \(A:=V\Lambda V^*\). Assume \(Y\in\mathbb{F}^{n\times m}\). Define \(X:=V^*Y\), and assume \(X(1:m,:)\) is invertible and all entries of \(S:=X(m+1:n,:)X(1:m,:)^{-1}\in\mathbb{F}^{(n-m)\times m}\) are nonzero.

Then there exists \(k_4=k_4(n,m,\Lambda,X)\in\mathbb{N}_+\) such that \(\forall k\geq k_4\),

\begin{itemize}
\item The \(m\) Ritz values of \(A\) projected on \(\mathrm{span}(A^kY)\) have distinct modulus; denote them by \(\{\theta_j\}_{j\in[m]}\) where \(|\theta_1|>\dots>|\theta_m|\).
\item Let \(W_k^e(:,j)\) be a unit (eigen) Ritz vector corresponding to \(\theta_j\). Then \(\forall j\in[m]\), we have
\begin{align*}
|V(:,i)^*W_k^e(:,j)|&=\left|\dfrac{\lambda_j-\lambda_{m+1}}{\lambda_i-\lambda_j}\right||S(1,i)||S(1,j)|\left|\dfrac{\lambda_{m+1}^2}{\lambda_i\lambda_j}\right|^k(1+o(1)),\quad i\in[m]\setminus\{j\};\\
|V(:,i)^*W_k^e(:,j)|&=|S(i-m,j)|\left|\dfrac{\lambda_i}{\lambda_j}\right|^k(1+o(1)),\quad i\in[n]\setminus[m];\\
1-|V(:,j)^*W_k^e(:,j)|&=\dfrac{1}{2}|S(1,j)|^2\left|\dfrac{\lambda_{m+1}}{\lambda_j}\right|^{2k}(1+o(1)).
\end{align*}
\end{itemize}
\end{corollary}
\begin{proof}
This follows by applying a change of orthonormal basis to the conclusion of Theorem \ref{theorem:withRR diagonal eigen}.
\end{proof}

We now prove Theorem~\ref{theorem:withRR diagonal eigen}.

\subsection{Notation}

\begin{itemize}
\item \(\Lambda_m:=\Lambda(1:m,1:m)\in\mathbb{F}^{m\times m}\), \(\Lambda_m^{\perp}:=\Lambda(m+1:n,m+1:n)\in\mathbb{F}^{(n-m)\times(n-m)}\).
\item \(\delta:=\min_{i\neq j}||\lambda_i|-|\lambda_j||>0\).
\item \(C:=\max\{1,|\lambda_{m+1}|\}\cdot\max_{p,q\in[m]}\{\sum_{r\in[n-m]}|S(r,p)S(r,q)|\}>0\).
\item \(M_0:=Cm(|\lambda_1|+2)\max\{1,2/\delta\}>0\).
\item \(k\in\mathbb{N}_+\).
\item \(H:=\Lambda_m^{\perp k}S\Lambda_m^{-k}\in\mathbb{F}^{(n-m)\times m}\), \(\Phi=\begin{bmatrix}
I_m\\
H
\end{bmatrix}\in\mathbb{F}^{n\times m}\), \(E:=H^*H\in\mathbb{F}^{m\times m}\), \(F:=H^*\Lambda_m^{\perp}H\in\mathbb{F}^{m\times m}\).
\end{itemize}

\subsection{Auxiliary lemmas and propositions}

\begin{definition}
Set
\[
\rho_{p,q}=\dfrac{|\lambda_{m+1}|^2}{|\lambda_p||\lambda_q|},\quad p,q\in[m];\qquad\rho_{\max}=\dfrac{|\lambda_{m+1}|^2}{|\lambda_m|^2}<1.
\]
\end{definition}

\begin{proposition}
\(\forall p,q,r\in[m]\), we have
\[
\rho_{p,q}\rho_{q,r}\leq\rho_{\max}\rho_{p,r}.
\]
\end{proposition}
\begin{proof}
\[
\rho_{p,q}\rho_{q,r}=\dfrac{|\lambda_{m+1}|^2}{|\lambda_p||\lambda_q|}\dfrac{|\lambda_{m+1}|^2}{|\lambda_q||\lambda_r|}=\dfrac{|\lambda_{m+1}|^2}{|\lambda_q|^2}\dfrac{|\lambda_{m+1}|^2}{|\lambda_p||\lambda_r|}\leq\rho_{\max}\rho_{p,r}.
\]
\end{proof}

\begin{proposition}
\begin{enumerate}
\item \(H(p,q)=S(p,q)\left(\dfrac{\lambda_{m+p}}{\lambda_q}\right)^k,\quad\forall(p,q)\in[n-m]\times[m]\).
\item \(E(p,q)=\sum_{r\in[n-m]}\overline{S(r,p)}S(r,q)\overline{\left(\dfrac{\lambda_{m+r}}{\lambda_p}\right)^k}\left(\dfrac{\lambda_{m+r}}{\lambda_q}\right)^k,\quad\forall(p,q)\in[m]^2\).
\item \(F(p,q)=\sum_{r\in[n-m]}\lambda_{m+r}\overline{S(r,p)}S(r,q)\overline{\left(\dfrac{\lambda_{m+r}}{\lambda_p}\right)^k}\left(\dfrac{\lambda_{m+r}}{\lambda_q}\right)^k,\quad\forall(p,q)\in[m]^2\).
\end{enumerate}
\end{proposition}

\begin{proposition}
\(\forall(p,q)\in[m]^2\), we have \(|E(p,q)|\leq C\rho_{p,q}^k\), \(|F(p,q)|\leq C\rho_{p,q}^k\).
\end{proposition}
\begin{proof}
\[
|E(p,q)|\leq\sum_{r\in[n-m]}|S(r,p)S(r,q)|\dfrac{|\lambda_{m+r}|^{2k}}{|\lambda_p|^k|\lambda_q|^k}\leq\left(\sum_{r\in[n-m]}|S(r,p)S(r,q)|\right)\rho_{p,q}^k\leq C\rho_{p,q}^k.
\]
\[
|F(p,q)|\leq\sum_{r\in[n-m]}|\lambda_{m+r}S(r,p)S(r,q)|\dfrac{|\lambda_{m+r}|^{2k}}{|\lambda_p|^k|\lambda_q|^k}\leq\left(|\lambda_{m+1}|\sum_{r\in[n-m]}|S(r,p)S(r,q)|\right)\rho_{p,q}^k\leq C\rho_{p,q}^k.
\]
\end{proof}

\subsection{Construction of the metric space \((B,d)\) and the map \(T\)}

Fix \(j\in[m]\). Fix \(y(j)=1\), where \(y\in\mathbb{F}^m\), and let \(\hat{y}\in\mathbb{F}^{m-1}\) denote \(y\) with its \(j\)-th component removed. Consider the set
\[
B=\left\{(\theta,\hat{y})\in\mathbb{F}\times\mathbb{F}^{m-1}\middle|\dfrac{|\theta-\lambda_j|}{\rho_{j,j}^k}\leq M_0,\max_{q\in[m]\setminus\{j\}}\dfrac{|y(q)|}{\rho_{j,q}^k}\leq M_0\right\}.
\]

Define metric \(d\) on \(B\):
\[
d((\theta_a,\hat{y}_a),(\theta_b,\hat{y}_b)):=\max\left\{\dfrac{|\theta_a-\theta_b|}{\rho_{j,j}^k},\max_{q\in[m]\setminus\{j\}}\dfrac{|y_a(q)-y_b(q)|}{\rho_{j,q}^k}\right\}.
\]

\begin{proposition}
\(d\) is indeed a metric, and the metric space \((B,d)\) is complete.
\end{proposition}
\begin{proof}
The proof is straightforward.
\end{proof}

There exists \(k_1\in\mathbb{N}_+\) such that, for \(k\geq k_1\), \(M_0\rho_{\max}^k<\delta/2\). Suppose \(k\geq k_1\) from now on.

Define \(T:B\to\mathbb{F}\times\mathbb{F}^{m-1}\) by \((\theta,\hat{y})\mapsto(\theta',\hat{y}')\).
\begin{align*}
\theta'&=\lambda_j-\theta\sum_{q\in[m]}E(j,q)y(q)+\sum_{q\in[m]}F(j,q)y(q),\\
y'(p)&=\dfrac{1}{\lambda_p-\theta}\left[\theta\sum_{q\in[m]}E(p,q)y(q)-\sum_{q\in[m]}F(p,q)y(q)\right],\quad p\in[m]\setminus\{j\}.
\end{align*}

\subsection{\(T(B)\subseteq B\)}

There exists \(k_2\geq k_1\) such that, for \(k\geq k_2\), \(M_0\rho_{\max}^k<1\). Suppose \(k\geq k_2\) from now on.

\begin{lemma}
\(\forall p\in[m]\), we have
\[
\left|\theta\sum_{q\in[m]}E(p,q)y(q)-\sum_{q\in[m]}F(p,q)y(q)\right|\leq(|\lambda_1|+2)mC\rho_{p,j}^k.
\]
\end{lemma}
\begin{proof}
\begin{align*}
&\left|\theta\sum_{q\in[m]}E(p,q)y(q)-\sum_{q\in[m]}F(p,q)y(q)\right|\\
&\leq|\theta|\sum_{q\in[m]}|E(p,q)||y(q)|+\sum_{q\in[m]}|F(p,q)||y(q)|\\
&\leq(|\lambda_j|+M_0\rho_{j,j}^k)\left(|E(p,j)|+\sum_{q\in[m]\setminus\{j\}}|E(p,q)||y(q)|\right)+|F(p,j)|+\sum_{q\in[m]\setminus\{j\}}|F(p,q)||y(q)|\\
&\leq(|\lambda_1|+2)\left(C\rho_{p,j}^k+\sum_{q\in[m]\setminus\{j\}}C\rho_{p,j}^kM_0\rho_{\max}^k\right)\\
&\leq(|\lambda_1|+2)mC\rho_{p,j}^k.
\end{align*}
\end{proof}

\begin{theorem}
\(T(B)\subseteq B\).
\end{theorem}
\begin{proof}
\(\forall (\theta,\hat{y})\in B\). By the preceding Lemma and the definition of \(M_0\),
\[
|\theta'-\lambda_j|=\left|-\theta\sum_{q\in[m]}E(j,q)y(q)+\sum_{q\in[m]}F(j,q)y(q)\right|\leq (|\lambda_1|+2)mC\rho_{j,j}^k\leq M_0\rho_{j,j}^k.
\]
\(\forall p\in[m]\setminus\{j\}\),
\[
|\lambda_p-\theta|\geq|\lambda_p-\lambda_j|-|\lambda_j-\theta|\geq\delta-\delta/2=\delta/2,
\]
by the preceding Lemma and the definition of \(M_0\),
\[
|y'(p)|=\dfrac{1}{|\lambda_p-\theta|}\left|\theta\sum_{q\in[m]}E(p,q)y(q)-\sum_{q\in[m]}F(p,q)y(q)\right|\leq\dfrac{2}{\delta}(|\lambda_1|+2)mC\rho_{p,j}^k\leq M_0\rho_{p,j}^k.
\]
Hence we have \(T(\theta,\hat{y})=(\theta',\hat{y}')\in B\). Therefore \(T(B)\subseteq B\).
\end{proof}

\subsection{\(T\) is a strict contraction}

\begin{lemma}
\(\forall (\theta_a,\hat{y}_a)\), \((\theta_b,\hat{y}_b)\in B\); \(\forall p\in[m]\), we have
\begin{align*}
&\left|\left[\theta_a\sum_{q\in[m]}E(p,q)y_a(q)-\sum_{q\in[m]}F(p,q)y_a(q)\right]-\left[\theta_b\sum_{q\in[m]}E(p,q)y_b(q)-\sum_{q\in[m]}F(p,q)y_b(q)\right]\right|\\
&\leq Cd(a,b)\rho_{p,j}^k\rho_{\max}^km(|\lambda_1|+3).
\end{align*}
\end{lemma}
\begin{proof}
\begin{align*}
&\left|\left[\theta_a\sum_{q\in[m]}E(p,q)y_a(q)-\sum_{q\in[m]}F(p,q)y_a(q)\right]-\left[\theta_b\sum_{q\in[m]}E(p,q)y_b(q)-\sum_{q\in[m]}F(p,q)y_b(q)\right]\right|\\
&\leq\left|\sum_{q\in[m]}E(p,q)\theta_ay_a(q)-\sum_{q\in[m]}E(p,q)\theta_by_b(q)\right|+\left|\sum_{q\in[m]}F(p,q)(y_a(q)-y_b(q))\right|\\
&\leq\sum_{q\in[m]}|E(p,q)||\theta_ay_a(q)-\theta_by_b(q)|+\sum_{q\in[m]}|F(p,q)||y_a(q)-y_b(q)|\\
&=|E(p,j)||\theta_a-\theta_b|+\sum_{q\in[m]\setminus\{j\}}|E(p,q)||\theta_ay_a(q)-\theta_by_b(q)|+\sum_{q\in[m]\setminus\{j\}}|F(p,q)||y_a(q)-y_b(q)|\\
&\leq C\rho_{p,j}^k\rho_{j,j}^kd(a,b)+\sum_{q\in[m]\setminus\{j\}}C\rho_{p,q}^k(|\theta_a||y_a(q)-y_b(q)|+|\theta_a-\theta_b||y_b(q)|)+\sum_{q\in[m]\setminus\{j\}}C\rho_{p,q}^k\rho_{j,q}^kd(a,b)\\
&\leq mC\rho_{p,j}^k\rho_{\max}^kd(a,b)+\sum_{q\in[m]\setminus\{j\}}C\rho_{p,q}^k((|\lambda_1|+1)\rho_{j,q}^kd(a,b)+\rho_{j,j}^kd(a,b))\\
&\leq Cd(a,b)\rho_{p,j}^k\rho_{\max}^km(|\lambda_1|+3).
\end{align*}
\end{proof}

There exists \(k_3\geq k_2\) such that, for \(k\geq k_3\), \(C\rho_{\max}^km(|\lambda_1|+3)\leq1/2\). Suppose \(k\geq k_3\) from now on.

\begin{proposition}
\(\forall(\theta_a,\hat{y}_a)\), \((\theta_b,\hat{y}_b)\in B\), we have \(\dfrac{|\theta_a'-\theta_b'|}{\rho_{j,j}^k}\leq\dfrac{1}{2}d(a,b)\).
\end{proposition}
\begin{proof}
\(\forall(\theta_a,\hat{y}_a)\), \((\theta_b,\hat{y}_b)\in B\). Set \((\theta_a',\hat{y}_a')=T(\theta_a,\hat{y}_a)\), \((\theta_b',\hat{y}_b')=T(\theta_b,\hat{y}_b)\). Then

\begin{align*}
&|\theta_a'-\theta_b'|\\
&=\left|\left(-\theta_a\sum_{q\in[m]}E(j,q)y_a(q)+\sum_{q\in[m]}F(j,q)y_a(q)\right)-\left(-\theta_b\sum_{q\in[m]}E(j,q)y_b(q)+\sum_{q\in[m]}F(j,q)y_b(q)\right)\right|\\
&\leq Cd(a,b)\rho_{j,j}^k\rho_{\max}^km(|\lambda_1|+3)\leq\dfrac{1}{2}\rho_{j,j}^kd(a,b),
\end{align*}

Hence
\[
\dfrac{|\theta_a'-\theta_b'|}{\rho_{j,j}^k}\leq\dfrac{1}{2}d(a,b).
\]
\end{proof}

There exists \(k_4\geq k_3\) such that, for \(k\geq k_4\), \(C\dfrac{8}{\delta^2}\rho_{\max}^km(|\lambda_1|+3)^2\leq1/2\). Suppose \(k\geq k_4\) from now on.

\begin{proposition}
\(\forall(\theta_a,\hat{y}_a)\), \((\theta_b,\hat{y}_b)\in B\); \(\forall p\in[m]\setminus\{j\}\), we have \(\dfrac{|y_a'(p)-y_b'(p)|}{\rho_{p,j}^k}\leq\dfrac{1}{2}d(a,b)\).
\end{proposition}
\begin{align*}
&|y_a'(p)-y_b'(p)|\\
&=\left|\dfrac{1}{\lambda_p-\theta_a}\left[\theta_a\sum_{q\in[m]}E(p,q)y_a(q)-\sum_{q\in[m]}F(p,q)y_a(q)\right]-\dfrac{1}{\lambda_p-\theta_b}\left[\theta_b\sum_{q\in[m]}E(p,q)y_b(q)-\sum_{q\in[m]}F(p,q)y_b(q)\right]\right|\\
&\leq\dfrac{4}{\delta^2}\left|(\lambda_p-\theta_b)\left[\theta_a\sum_{q\in[m]}E(p,q)y_a(q)-\sum_{q\in[m]}F(p,q)y_a(q)\right]-(\lambda_p-\theta_a)\left[\theta_b\sum_{q\in[m]}E(p,q)y_b(q)-\sum_{q\in[m]}F(p,q)y_b(q)\right]\right|\\
&\leq\dfrac{4}{\delta^2}|\theta_a-\theta_b|\left|\theta_a\sum_{q\in[m]}E(p,q)y_a(q)-\sum_{q\in[m]}F(p,q)y_a(q)\right|\\
&\quad+\dfrac{4}{\delta^2}|\lambda_p-\theta_a|\left|\left[\theta_a\sum_{q\in[m]}E(p,q)y_a(q)-\sum_{q\in[m]}F(p,q)y_a(q)\right]-\left[\theta_b\sum_{q\in[m]}E(p,q)y_b(q)-\sum_{q\in[m]}F(p,q)y_b(q)\right]\right|\\
&\leq\dfrac{4}{\delta^2}\rho_{j,j}^kd(a,b)(|\lambda_1|+2)mC\rho_{p,j}^k+\dfrac{4}{\delta^2}(2|\lambda_1|+1)Cd(a,b)\rho_{p,j}^k\rho_{\max}^km(|\lambda_1|+3)\\
&\leq\rho_{p,j}^kd(a,b)C\dfrac{8}{\delta^2}\rho_{\max}^km(|\lambda_1|+3)^2\\
&\leq\dfrac{1}{2}\rho_{p,j}^kd(a,b).
\end{align*}

So
\[
\dfrac{|y_a'(p)-y_b'(p)|}{\rho_{p,j}^k}\leq\dfrac{1}{2}d(a,b).
\]

\begin{theorem}
\(\forall(\theta_a,\hat{y}_a)\), \((\theta_b,\hat{y}_b)\in B\), we have
\(d(a',b')\leq\dfrac{1}{2}d(a,b)\).
\end{theorem}
\begin{proof}
\[
d(a',b')=\max\left\{\dfrac{|\theta_a'-\theta_b'|}{\rho_{j,j}^k},\max_{q\in[m]\setminus\{j\}}\dfrac{|y_a'(q)-y_b'(q)|}{\rho_{j,q}^k}\right\}\leq\dfrac{1}{2}d(a,b).
\]
\end{proof}

\subsection{Application of the Banach fixed point theorem}

So far we have proved that the following hold:
\begin{enumerate}
\item \((B,d)\) is a nonempty complete metric space;
\item \(T(B)\subseteq B\);
\item \(\forall a,b\in B\), \(d(T(a),T(b))\leq\dfrac{1}{2}d(a,b)\).
\end{enumerate}
By the Banach fixed point theorem, \(T\) has a unique fixed point in \(B\). Denote it by \((\theta_j,\hat{y}_j)\). By \(T(\theta_j,\hat{y}_j)=(\theta_j,\hat{y}_j)\), we know that\((\Lambda_m-\theta_jI_m)y_j=(\theta_jE-F)y_j\), which gives \((\Lambda_m+F)y_j=(I_m+E)y_j\theta_j\), and hence \(\Phi^*\Lambda\Phi y_j=\Phi^*\Phi y_j\theta_j\). Note that
\[
\Phi=\begin{bmatrix}
I_m\\
H
\end{bmatrix}=\begin{bmatrix}
I_m\\
\Lambda_m^{\perp k}S\Lambda_m^{-k}
\end{bmatrix}=\begin{bmatrix}
\Lambda_m^kX(1:m,:)\\
\Lambda_m^{\perp k}X(m+1:n,:)
\end{bmatrix}X(1:m,:)^{-1}\Lambda_m^{-k}=\Lambda^kXX(1:m,:)^{-1}\Lambda_m^{-k},
\]
thus, \((\theta_j,\Phi y_j)\) is a Ritz pair of \(\Lambda\) projected on \(\mathrm{span}(\Lambda^kX)\).

\((\theta_j,\hat{y}_j)\in B\) implies \(|\theta_j-\lambda_j|\leq M_0\rho_{\max}^k<\delta/2\). Therefore, \(\forall j\in[m-1]\),
\[
|\theta_j|>|\lambda_j|-\delta/2=|\lambda_{j+1}|+|\lambda_j|-|\lambda_{j+1}|-\delta/2\geq |\lambda_{j+1}|+\delta/2>|\theta_{j+1}|.
\]
Therefore \(\{(\theta_j,\Phi y_j)\}_{j\in[m]}\) are exactly the \(m\) Ritz pairs of \(\Lambda\) projected on \(\mathrm{span}(\Lambda^kX)\), with \(|\theta_1|>\dots>|\theta_m|\).

Since \((\theta_j,\hat{y}_j)\in B\), we have the following propositions.

\begin{proposition}
\(\forall j\in[m]\), we have
\[
|\theta_j-\lambda_j|\leq M_0\rho_{j,j}^k.
\]
\end{proposition}

\begin{proposition}
\(\forall j\in[m]\); \(\forall q\in[m]\setminus\{j\}\), we have
\[
|y_j(q)|\leq M_0\rho_{j,q}^k.
\]
\end{proposition}

\subsection{Formulas of \(y_j(p)\)}

\(\forall j\in[m]\). By the equation \((\Lambda_m-\theta_jI_m)y_j=(\theta_jE-F)y_j\), \(\forall p\in[m]\setminus\{j\}\), we have
\begin{align*}
&y_j(p)\\
&=\dfrac{1}{\lambda_p-\theta_j}\left[\theta_j(Ey_j)(p)-(Fy_j)(p)\right]\\
&=\dfrac{1}{\lambda_p-\theta_j}\left[\theta_j\sum_{q\in[m]}E(p,q)y_j(q)-\sum_{q\in[m]}F(p,q)y_j(q)\right]\\
&=\dfrac{1}{\lambda_p-(\lambda_j+o(1))}\left[(\lambda_j+o(1))\left(E(p,j)+\sum_{q\in[m]\setminus\{j\}}E(p,q)y_j(q)\right)-\left(F(p,j)+\sum_{q\in[m]\setminus\{j\}}F(p,q)y_j(q)\right)\right]\\
&=\dfrac{1+o(1)}{\lambda_p-\lambda_j}\left[(\lambda_j+o(1))\overline{S(1,p)}S(1,j)\overline{\left(\dfrac{\lambda_{m+1}}{\lambda_p}\right)^k}\left(\dfrac{\lambda_{m+1}}{\lambda_j}\right)^k-(\lambda_{m+1}+o(1))\overline{S(1,p)}S(1,j)\overline{\left(\dfrac{\lambda_{m+1}}{\lambda_p}\right)^k}\left(\dfrac{\lambda_{m+1}}{\lambda_j}\right)^k\right]\\
&=\dfrac{1+o(1)}{\lambda_p-\lambda_j}\left[(\lambda_j-\lambda_{m+1})\overline{S(1,p)}S(1,j)\overline{\left(\dfrac{\lambda_{m+1}}{\lambda_p}\right)^k}\left(\dfrac{\lambda_{m+1}}{\lambda_j}\right)^k\right]\\
&=\dfrac{\lambda_j-\lambda_{m+1}}{\lambda_p-\lambda_j}\overline{S(1,p)}S(1,j)\overline{\left(\dfrac{\lambda_{m+1}}{\lambda_p}\right)^k}\left(\dfrac{\lambda_{m+1}}{\lambda_j}\right)^k(1+o(1)).
\end{align*}

\subsection{Formulas of \((\Phi y_j)(p)\)}

\(\forall j\in[m]\). The \(j\)-th Ritz vector is
\[
\Phi y_j=\begin{bmatrix}
I_m\\
H
\end{bmatrix}y_j=\begin{bmatrix}
y_j\\
Hy_j
\end{bmatrix}.
\]
Therefore,
\[
(\Phi y_j)(j)=y_j(j)=1,
\]
\(\forall p\in[m]\setminus\{j\}\), we have
\[
(\Phi y_j)(p)=y_j(p)=\dfrac{\lambda_j-\lambda_{m+1}}{\lambda_p-\lambda_j}\overline{S(1,p)}S(1,j)\overline{\left(\dfrac{\lambda_{m+1}}{\lambda_p}\right)^k}\left(\dfrac{\lambda_{m+1}}{\lambda_j}\right)^k(1+o(1)),
\]
\(\forall p\in[n-m]\), we have
\begin{align*}
(\Phi y_j)(m+p)&=(Hy_j)(p)=\sum_{q\in[m]}H(p,q)y_j(q)=H(p,j)+\sum_{q\in[m]\setminus\{j\}}H(p,q)y_j(q)\\
&=S(p,j)\left(\dfrac{\lambda_{m+p}}{\lambda_j}\right)^k(1+o(1)).
\end{align*}

\subsection{Completion of the proof}

The Ritz pairs with unit Ritz vectors are given by \(\left\{\left(\theta_j,\dfrac{\Phi y_j}{\|\Phi y_j\|_2}\right)\right\}_{j\in[m]}\). \(\forall j\in[m]\), define \(Z_k^e(:,j):=\dfrac{\Phi y_j}{\|\Phi y_j\|_2}\).

\(\forall j\in[m]\). We have
\begin{align*}
\|\Phi y_j\|_2&=\left(\|\Phi y_j\|_2^2\right)^{1/2}=\left(|(\Phi y_j)(j)|^2+\sum_{p\in[m]\setminus\{j\}}|(\Phi y_j)(p)|^2+\sum_{p\in[n-m]}|(\Phi y_j)(m+p)|^2\right)^{1/2}\\
&=\left(1+o(1)\right)^{1/2}=1+o(1).
\end{align*}
Therefore \(\forall i\in[m]\setminus\{j\}\), we have
\[
|Z_k^e(i,j)|=\dfrac{|(\Phi y_j)(i)|}{\|\Phi y_j\|_2}=\left|\dfrac{\lambda_j-\lambda_{m+1}}{\lambda_i-\lambda_j}\right||S(1,i)||S(1,j)|\left|\dfrac{\lambda_{m+1}^2}{\lambda_i\lambda_j}\right|^k(1+o(1)).
\]
\(\forall i\in[n]\setminus[m]\), we have
\[
|Z_k^e(i,j)|=\dfrac{|(\Phi y_j)(i)|}{\|\Phi y_j\|_2}=|S(i-m,j)|\left|\dfrac{\lambda_i}{\lambda_j}\right|^k(1+o(1)).
\]
And
\[
|Z_k^e(j,j)|=\dfrac{|(\Phi y_j)(j)|}{\|\Phi y_j\|_2}=(\|\Phi y_j\|_2^2)^{-1/2}=1-\dfrac{1}{2}|S(1,j)|^2\left|\dfrac{\lambda_{m+1}}{\lambda_j}\right|^{2k}(1+o(1)).
\]

All unit (eigen) Ritz vectors corresponding to \(\theta_j\) are of the form \(Z_k^e(:,j)e^{\imath\alpha}\) for \(\alpha\in\mathbb{R}\). It is clear that \(|(Z_k^e(:,j)e^{\imath\alpha})(i)|\) is independent of \(\alpha\). This completes the proof of Theorem~\ref{theorem:withRR diagonal eigen}.

\section{Subspace iteration with the Rayleigh--Ritz procedure II}\label{section:withRR diagonal Schur}

In this section, we derive the element-wise convergence behavior of subspace iteration with the (Schur) Rayleigh--Ritz procedure, where the matrix \(\Lambda\) for left multiplication to be diagonal.

\begin{theorem}\label{theorem:withRR diagonal Schur}
Suppose \(\mathbb{F}=\mathbb{R}\) or \(\mathbb{C}\), \(n\), \(m\in\mathbb{N}_+\) with \(n>m\), and \(\Lambda=\mathrm{diag}(\lambda_1,\dots,\lambda_n)\in\mathbb{F}^{n\times n}\), with \(|\lambda_1|>\dots>|\lambda_n|>0\). Assume \(X\in\mathbb{F}^{n\times m}\) is such that \(X(1:m,:)\) is invertible and all entries of \(S:=X(m+1:n,:)X(1:m,:)^{-1}\in\mathbb{F}^{(n-m)\times m}\) are nonzero. 
	
Then there exists \(k_4=k_4(n,m,\Lambda,X)\in\mathbb{N}_+\) such that \(\forall k\geq k_4\),
	
\begin{itemize}
\item The \(m\) Ritz values of \(\Lambda\) projected on \(\mathrm{span}(\Lambda^kX)\) have distinct modulus; denote them by \(\{\theta_j\}_{j\in[m]}\) where \(|\theta_1|>\dots>|\theta_m|\).
\item Let \(Z_k^s(:,j)\) be a unit (Schur) Ritz vector corresponding to \(\theta_j\). Then \(\forall j\in[m]\), we have
\begin{align*}
|Z_k^s(i,j)|&=\left|\dfrac{\lambda_j-\lambda_{m+1}}{\lambda_i-\lambda_j}\right||S(1,i)||S(1,j)|\left|\dfrac{\lambda_{m+1}^2}{\lambda_i\lambda_j}\right|^k(1+o(1)),\quad i\in[m]\setminus\{j\},\\
|Z_k^s(i,j)|&=|S(i-m,j)|\left|\dfrac{\lambda_i}{\lambda_j}\right|^k(1+o(1)),\quad i\in[n]\setminus[m],\\
1-|Z_k^s(j,j)|&=\dfrac{1}{2}|S(1,j)|^2\left|\dfrac{\lambda_{m+1}}{\lambda_j}\right|^{2k}(1+o(1)).
\end{align*}
\end{itemize}
\end{theorem}
\begin{proof}
See the rest of this section.
\end{proof}

\begin{corollary}
Suppose \(\mathbb{F}=\mathbb{R}\) or \(\mathbb{C}\), \(n\), \(m\in\mathbb{N}_+\) with \(n>m\), and \(\Lambda=\mathrm{diag}(\lambda_1,\dots,\lambda_n)\in\mathbb{F}^{n\times n}\), with \(|\lambda_1|>\dots>|\lambda_n|>0\). Assume \(V\in\mathbb{F}^{n\times n}\) satisfies \(VV^*=V^*V=I_n\). Define \(A:=V\Lambda V^*\). Assume \(Y\in\mathbb{F}^{n\times m}\). Define \(X:=V^*Y\), and assume \(X(1:m,:)\) is invertible and all entries of \(S:=X(m+1:n,:)X(1:m,:)^{-1}\in\mathbb{F}^{(n-m)\times m}\) are nonzero.
	
Then there exists \(k_4=k_4(n,m,\Lambda,X)\in\mathbb{N}_+\) such that \(\forall k\geq k_4\),
	
\begin{itemize}
\item The \(m\) Ritz values of \(A\) projected on \(\mathrm{span}(A^kY)\) have distinct modulus; denote them by \(\{\theta_j\}_{j\in[m]}\) where \(|\theta_1|>\dots>|\theta_m|\).
\item Let \(W_k^s(:,j)\) be a unit (Schur) Ritz vector corresponding to \(\theta_j\). Then \(\forall j\in[m]\), we have
\begin{align*}
|V(:,i)^*W_k^s(:,j)|&=\left|\dfrac{\lambda_j-\lambda_{m+1}}{\lambda_i-\lambda_j}\right||S(1,i)||S(1,j)|\left|\dfrac{\lambda_{m+1}^2}{\lambda_i\lambda_j}\right|^k(1+o(1)),\quad i\in[m]\setminus\{j\};\\
|V(:,i)^*W_k^s(:,j)|&=|S(i-m,j)|\left|\dfrac{\lambda_i}{\lambda_j}\right|^k(1+o(1)),\quad i\in[n]\setminus[m];\\
1-|V(:,j)^*W_k^s(:,j)|&=\dfrac{1}{2}|S(1,j)|^2\left|\dfrac{\lambda_{m+1}}{\lambda_j}\right|^{2k}(1+o(1)).
\end{align*}
\end{itemize}
\end{corollary}
\begin{proof}
This follows by applying a change of orthonormal basis to the conclusion of Theorem \ref{theorem:withRR diagonal Schur}.
\end{proof}

We now prove Theorem~\ref{theorem:withRR diagonal Schur}. Since the conditions of Theorem~\ref{theorem:withRR diagonal Schur} and Theorem~\ref{theorem:withRR diagonal eigen} are exactly the same, the notation, procedure, and conclusions of Theorem~\ref{theorem:withRR diagonal eigen} can all be used here.

\subsection{Gram--Schmidt orthogonalization of the columns of \(Z_k^e\)}

By the relation of Schur decomposition and eigen decomposition, we know that
\[
Z_k^s=Q(Z_k^e).
\]
Therefore, just define \(Z_k^s\) as follows: \(\forall j\in[m]\),
\[
V_k(:,j):=Z_k^e(:,j)-\sum_{l\in[j-1]}\langle Z_k^e(:,j),Z_k^s(:,l)\rangle Z_k^s(:,l),\quad Z_k^s(:,j):=\dfrac{V_k(:,j)}{\|V_k(:,j)\|_2}.
\]
Claim: \(\forall j\in[m]\), we have
\begin{align*}
|Z_k^s(i,j)|&=O(\rho_{i,j}^k),&&i\in[m]\setminus\{j\},\\
|Z_k^s(i,j)|&=O\left(\left|\dfrac{\lambda_i}{\lambda_j}\right|^k\right),&&i\in[n]\setminus[m],\\
|Z_k^s(i,j)-Z_k^e(i,j)|&=o(\rho_{i,j}^k),&&i\in[m]\setminus[j],\\
|Z_k^s(j,j)-1|&=o(1),&& \\
|Z_k^s(i,j)-Z_k^e(i,j)|&=o\left(\left|\dfrac{\lambda_i}{\lambda_j}\right|^k\right),&&i\in[n]\setminus[m].
\end{align*}
We prove the claim by induction on \(j\).

\subsection{Base case}

Consider the case \(j=1\). We have \(V_k(:,1)=Z_k^e(:,1)\) and \(Z_k^s(:,1)=Z_k^e(:,1)\). Therefore by the results in section~\ref{section:withRR diagonal eigen},
\[
|Z_k^s(i,1)|=|Z_k^e(i,1)|=O(\rho_{i,1}^k),\quad\forall i\in[m]\setminus\{1\},
\]
\[
|Z_k^s(i,1)|=|Z_k^e(i,1)|=O\left(\left|\dfrac{\lambda_i}{\lambda_1}\right|^k\right),\quad\forall i\in[n]\setminus[m],
\]
\[
|Z_k^s(i,1)-Z_k^e(i,1)|=0=o(\rho_{i,1}^k),\quad\forall i\in[m]\setminus[1],
\]
\[
|Z_k^s(1,1)-1|=\left|\dfrac{(\Phi y_1)(1)}{\|\Phi y_1\|_2}-1\right|=\left|\dfrac{1}{1+o(1)}-1\right|=o(1),
\]
\[
|Z_k^s(i,1)-Z_k^e(i,1)|=0=o\left(\left|\dfrac{\lambda_i}{\lambda_1}\right|^k\right),\quad\forall i\in[n]\setminus[m].
\]

\subsection{Inductive step}

Assume the claim holds for \(1,2,\dots,j-1\). We prove it for \(j\).

\subsubsection{The bound of \(|c_{l,j}|\)}

\(\forall l\in[j-1]\), define \(c_{l,j}:=\langle Z_k^e(:,j),Z_k^s(:,l)\rangle\). Hence we have
\[
V_k(:,j)=Z_k^e(:,j)-\sum_{l\in[j-1]}c_{l,j}Z_k^s(:,l).
\]
As a consequence, \(\forall i\in[n]\),
\[
V_k(i,j)=Z_k^e(i,j)-\sum_{l\in[j-1]}c_{l,j}Z_k^s(i,l).
\]
\(\forall l\in[j-1]\), direct calculation yields
\begin{align*}
|c_{l,j}|&=\left|\sum_{r\in[n]}Z_k^e(r,j)\overline{Z_k^s(r,l)}\right|\leq\sum_{r\in[n]}|Z_k^e(r,j)||Z_k^s(r,l)|\\
&=|Z_k^e(l,j)||Z_k^s(l,l)|+|Z_k^e(j,j)||Z_k^s(j,l)|+\sum_{r\in[n]\setminus\{l,j\}}|Z_k^e(r,j)||Z_k^s(r,l)|\\
&\leq O(\rho_{l,j}^k)O(1)+O(1)O(\rho_{j,l}^k)+\sum_{r\in[m]\setminus\{l,j\}}O(\rho_{r,j}^k\rho_{r,l}^k)+\sum_{r\in[n]\setminus[m]}O\left(\left|\dfrac{\lambda_r}{\lambda_j}\right|^k\left|\dfrac{\lambda_r}{\lambda_l}\right|^k\right)\\
&\leq O(\rho_{l,j}^k)+O(\rho_{l,j}^k)+O(\rho_{j,l}^k)+O(\rho_{j,l}^k)=O(\rho_{l,j}^k).
\end{align*}

\subsubsection{Inequalities about \(V_k(i,j)\)}

\(\forall i\in[j-1]\), we have
\begin{align*}
|V_k(i,j)|&\leq|Z_k^e(i,j)|+\sum_{l\in[j-1]}|c_{l,j}||Z_k^s(i,l)|\\
&\leq|Z_k^e(i,j)|+|c_{i,j}||Z_k^s(i,i)|+\sum_{l\in[j-1]\setminus\{i\}}|c_{l,j}||Z_k^s(i,l)|\\
&=O(\rho_{i,j}^k)+O(\rho_{i,j}^k)O(1)+\sum_{l\in[j-1]\setminus\{i\}}O(\rho_{l,j}^k)O(\rho_{i,l}^k)\\
&\leq O(\rho_{i,j}^k).
\end{align*}

\(\forall i\in[m]\setminus[j]\),
\begin{align*}
|V_k(i,j)-Z_k^e(i,j)|&\leq\sum_{l\in[j-1]}|c_{l,j}||Z_k^s(i,l)|\leq\sum_{l\in[j-1]}O(\rho_{l,j}^k)O(\rho_{i,l}^k)\\
&\leq O(\rho_{i,j}^k\rho_{\max}^k)=o(\rho_{i,j}^k).
\end{align*}
And hence
\[
|V_k(i,j)|\leq|V_k(i,j)-Z_k^e(i,j)|+|Z_k^e(i,j)|\leq o(\rho_{i,j}^k)+O(\rho_{i,j}^k)=O(\rho_{i,j}^k).
\]

\(\forall i\in[n]\setminus[m]\),
\begin{align*}
|V_k(i,j)-Z_k^e(i,j)|&\leq\sum_{l\in[j-1]}|c_{l,j}||Z_k^s(i,l)|\leq\sum_{l\in[j-1]}O(\rho_{l,j}^k)O\left(\left|\dfrac{\lambda_i}{\lambda_l}\right|^k\right)\\
&\leq O\left(\rho_{\max}^k\left|\dfrac{\lambda_i}{\lambda_j}\right|^k\right)=o\left(\left|\dfrac{\lambda_i}{\lambda_j}\right|^k\right).
\end{align*}
And hence
\[
|V_k(i,j)|\leq|V_k(i,j)-Z_k^e(i,j)|+|Z_k^e(i,j)|\leq o\left(\left|\dfrac{\lambda_i}{\lambda_j}\right|^k\right)+O\left(\left|\dfrac{\lambda_i}{\lambda_j}\right|^k\right)=O\left(\left|\dfrac{\lambda_i}{\lambda_j}\right|^k\right).
\]

\subsubsection{Inequalities about \(Z_k^s(i,j)\)}

By the definition of \(V_k(:,j)\),
\[
\|V_k(:,j)\|_2^2=\|Z_k^e(:,j)\|_2^2-\sum_{l\in[j-1]}|c_{l,j}|^2=1-\sum_{l\in[j-1]}O(\rho_{l,j}^{2k})=1+o(1),
\]
so \(\|V_k(:,j)\|_2=1+o(1)\). Now we will verify the five formulas in the claim hold for \(j\).

\textbf{The first line.} \(\forall i\in[m]\setminus\{j\}\),
\[
|Z_k^s(i,j)|=\dfrac{|V_k(i,j)|}{\|V_k(:,j)\|_2}=O(\rho_{i,j}^k)O(1)=O(\rho_{i,j}^k).
\]
\textbf{The second line.} \(\forall i\in[n]\setminus[m]\),
\[
|Z_k^s(i,j)|=\dfrac{|V_k(i,j)|}{\|V_k(:,j)\|_2}=O\left(\left|\dfrac{\lambda_i}{\lambda_j}\right|^k\right)O(1)=O\left(\left|\dfrac{\lambda_i}{\lambda_j}\right|^k\right).
\]
\textbf{The third line.} \(\forall i\in[m]\setminus[j]\),
\begin{align*}
|Z_k^s(i,j)-Z_k^e(i,j)|&\leq|V_k(i,j)|\left|\dfrac{1}{\|V_k(:,j)\|_2}-1\right|+|V_k(i,j)-Z_k^e(i,j)|\\
&\leq O(\rho_{i,j}^k)o(1)+o(\rho_{i,j}^k)=o(\rho_{i,j}^k).
\end{align*}
\textbf{The fourth line.} Since
\[
V_k(j,j)=Z_k^e(j,j)-\sum_{l\in[j-1]}c_{l,j}Z_k^s(j,l)=1+o(1)-\sum_{l\in[j-1]}O(\rho_{l,j}^k)O(\rho_{j,l}^k)=1+o(1),
\]
we have
\begin{align*}
|Z_k^s(j,j)-1|&\leq|V_k(j,j)|\left|\dfrac{1}{\|V_k(:,j)\|_2}-1\right|+|V_k(j,j)-1|\\
&\leq O(1)o(1)+o(1)=o(1).
\end{align*}
\textbf{The fifth line.} \(\forall i\in[n]\setminus[m]\),
\begin{align*}
|Z_k^s(i,j)-Z_k^e(i,j)|&\leq|V_k(i,j)|\left|\dfrac{1}{\|V_k(:,j)\|_2}-1\right|+|V_k(i,j)-Z_k^e(i,j)|\\
&\leq O\left(\left|\dfrac{\lambda_i}{\lambda_j}\right|^k\right)o(1)+o\left(\left|\dfrac{\lambda_i}{\lambda_j}\right|^k\right)=o\left(\left|\dfrac{\lambda_i}{\lambda_j}\right|^k\right).
\end{align*}

This is the claim for \(j\).

\subsection{Completing the proof of the claim}

By the base case and the inductive step, the claim holds for every \(j\in[m]\).

\subsection{The entries above the diagonal}

Put
\[
A:=Z_k^s(1:m,:)-I_m\in\mathbb{F}^{m\times m},\qquad B:=Z_k^s(m+1:n,:)\in\mathbb{F}^{(n-m)\times m}.
\]
By \((Z_k^s)^*Z_k^s=I_m\), we know that \(A^*A+A^*+A+B^*B=0\). Namely we have
\[
A(i,j)=-\overline{A(j,i)}-(A^*A)(i,j)-(B^*B)(i,j),\qquad\forall(i,j)\in[m]^2.
\]
Now fix \((i,j)\in[m]^2\) such that \(i<j\).

Direct calculation yields
\begin{align*}
|(A^*A)(i,j)|&\leq|A(i,i)||A(i,j)|+|A(j,i)||A(j,j)|+\sum_{l\in[m]\setminus\{i,j\}}|A(l,i)||A(l,j)|\\
&\leq o(1)O(\rho_{i,j}^k)+O(\rho_{i,j}^k)o(1)+\sum_{l\in[m]\setminus\{i,j\}}O(\rho_{l,i}^k)O(\rho_{l,j}^k)\\
&\leq o(\rho_{i,j}^k)+o(\rho_{i,j}^k)+o(\rho_{i,j}^k)=o(\rho_{i,j}^k).
\end{align*}
Given that
\[
B(p,l)=Z_k^e(m+p,l)+o\left(\left|\dfrac{\lambda_{m+p}}{\lambda_l}\right|^k\right)=S(p,l)\left(\dfrac{\lambda_{m+p}}{\lambda_l}\right)^k(1+o(1)),\quad\forall(p,l)\in[n-m]\times [m],
\]
we have
\[
(B^*B)(i,j)=\sum_{p\in[n-m]}\overline{B(p,i)}B(p,j)=\overline{S(1,i)}S(1,j)\overline{\left(\dfrac{\lambda_{m+1}}{\lambda_i}\right)^k}\left(\dfrac{\lambda_{m+1}}{\lambda_j}\right)^k(1+o(1)).
\]
Now we can calculate \(A(i,j)\):
\begin{align*}
&A(i,j)\\
&=-\overline{A(j,i)}-(A^*A)(i,j)-(B^*B)(i,j)\\
&=-\overline{Z_k^s(j,i)}-o(\rho_{i,j}^k)-\overline{S(1,i)}S(1,j)\overline{\left(\dfrac{\lambda_{m+1}}{\lambda_i}\right)^k}\left(\dfrac{\lambda_{m+1}}{\lambda_j}\right)^k\\
&=-\overline{Z_k^e(j,i)}-o(\rho_{i,j}^k)-\overline{S(1,i)}S(1,j)\overline{\left(\dfrac{\lambda_{m+1}}{\lambda_i}\right)^k}\left(\dfrac{\lambda_{m+1}}{\lambda_j}\right)^k\\
&=-\overline{\left(\dfrac{\lambda_i-\lambda_{m+1}}{\lambda_j-\lambda_i}\right)}\overline{S(1,i)}S(1,j)\overline{\left(\dfrac{\lambda_{m+1}}{\lambda_i}\right)^k}\left(\dfrac{\lambda_{m+1}}{\lambda_j}\right)^k-o(\rho_{i,j}^k)-\overline{S(1,i)}S(1,j)\overline{\left(\dfrac{\lambda_{m+1}}{\lambda_i}\right)^k}\left(\dfrac{\lambda_{m+1}}{\lambda_j}\right)^k\\
&=\overline{\left(\dfrac{\lambda_j-\lambda_{m+1}}{\lambda_i-\lambda_j}\right)}\overline{S(1,i)}S(1,j)\overline{\left(\dfrac{\lambda_{m+1}}{\lambda_i}\right)^k}\left(\dfrac{\lambda_{m+1}}{\lambda_j}\right)^k(1+o(1)).
\end{align*}
Therefore
\[
|A(i,j)|=\left|\dfrac{\lambda_j-\lambda_{m+1}}{\lambda_i-\lambda_j}\right||S(1,i)||S(1,j)|\left|\dfrac{\lambda_{m+1}^2}{\lambda_i\lambda_j}\right|^k(1+o(1)).
\]

\subsection{Completion of the proof}

\(\forall j\in[m]\).

If \(i\in[j-1]\), we have
\[
|Z_k^s(i,j)|=|A(i,j)|=\left|\dfrac{\lambda_j-\lambda_{m+1}}{\lambda_i-\lambda_j}\right||S(1,i)||S(1,j)|\left|\dfrac{\lambda_{m+1}^2}{\lambda_i\lambda_j}\right|^k(1+o(1)).
\]
If \(i\in[m]\setminus[j]\), by the third line of the claim,
\[
|Z_k^s(i,j)|=|Z_k^e(i,j)+o(\rho_{i,j}^k)|=\left|\dfrac{\lambda_j-\lambda_{m+1}}{\lambda_i-\lambda_j}\right||S(1,i)||S(1,j)|\left|\dfrac{\lambda_{m+1}^2}{\lambda_i\lambda_j}\right|^k(1+o(1)).
\]
If \(i\in[n]\setminus[m]\), by the fifth line of the claim,
\[
|Z_k^s(i,j)|=\left|Z_k^e(i,j)+o\left(\left|\dfrac{\lambda_i}{\lambda_j}\right|^k\right)\right|=|S(i-m,j)|\left|\dfrac{\lambda_i}{\lambda_j}\right|^k(1+o(1)).
\]
And for \(i=j\), since
\[
1-|Z_k^s(j,j)|^2=\sum_{i\in[m]\setminus\{j\}}|Z_k^s(i,j)|^2+\sum_{i\in[n]\setminus[m]}|Z_k^s(i,j)|^2=|S(1,j)|^2\left|\dfrac{\lambda_{m+1}}{\lambda_j}\right|^{2k}(1+o(1)),
\]
we have
\[
1-|Z_k^s(j,j)|=\dfrac{1-|Z_k^s(j,j)|^2}{1+|Z_k^s(j,j)|}=\dfrac{1}{2}|S(1,j)|^2\left|\dfrac{\lambda_{m+1}}{\lambda_j}\right|^{2k}(1+o(1)).
\]
All unit (Schur) Ritz vectors corresponding to \(\theta_j\) are of the form \(Z_k^s(:,j)e^{\imath\alpha}\) for \(\alpha\in\mathbb{R}\). It is clear that \(|(Z_k^s(:,j)e^{\imath\alpha})(i)|\) is independent of \(\alpha\). This completes the proof of Theorem \ref{theorem:withRR diagonal Schur}.

\section{Conclusion}\label{section:conclusion}

In this paper, we have studied the element-wise convergence behavior of subspace iteration for a unitarily diagonalizable matrix
\[
A = V \Lambda V^*, \qquad \Lambda = \operatorname{diag}(\lambda_1,\dots,\lambda_n),
\]
with \(|\lambda_1| > \cdots > |\lambda_n| > 0\). All results are established for both \(\mathbb{F} = \mathbb{R}\) and \(\mathbb{F} = \mathbb{C}\). All element-wise statements below are understood in the eigenvector coordinate system associated with \(V\), that is, we consider the entries of \(V^*\tilde Q_k\) and \(V^*W_k\), rather than the entries of \(\tilde Q_k\) and \(W_k\) in the standard basis.

For subspace iteration without the Rayleigh--Ritz procedure, let \(\tilde Q_k \tilde R_k = A^k Y\) be the QR factorization. Under mild assumptions of \(Y\), the entries of \(V^*\tilde Q_k\) satisfy
\[
|V(:,i)^*\tilde Q_k(:,j)|=\Theta \left(\left|\frac{\lambda_{\max\{i,j\}}}{\lambda_{\min\{i,j\}}}\right|^k\right), \quad (i,j)\in[n]\times [m],\quad i\neq j,
\]
\[
1-|V(:,j)^*\tilde Q_k(:,j)|=\Theta\left( \max\left\{\left|\frac{\lambda_j}{\lambda_{j-1}}\right|, \left|\frac{\lambda_{j+1}}{\lambda_j}\right|\right\}^{2k}\right),\quad j\in[m].
\]

For subspace iteration with the Rayleigh--Ritz procedure, let \(W_k\) denote the Ritz vectors of \(A\) projected on \(\mathrm{span}(A^kY)\), with Ritz values sorted in the descending order by magnitude. Under mild assumptions of \(Y\), the entries of \(V^*W_k\) satisfy
\[
|V(:,i)^*W_k(:,j)|=\Theta \left(\left|\frac{\lambda_i}{\lambda_j}\right|^k\right),\quad (i,j)\in([n]\setminus[m])\times[m],
\]
\[
|V(:,i)^*W_k(:,j)|=\Theta \left(\left|\frac{\lambda_{m+1}^2}{\lambda_i\lambda_j}\right|^k\right), \quad (i,j)\in[m]\times[m],\quad i\neq j,
\]
and
\[
1-|V(:,j)^*W_k(:,j)|=\Theta \left(\left|\frac{\lambda_{m+1}}{\lambda_j}\right|^{2k}\right), \quad j\in[m].
\]
These results provide an element-wise description of the convergence of subspace iteration, both without and with the Rayleigh--Ritz procedure, in the eigenvector coordinate system.

\appendix

\section{Proof of Lemma \ref{lemma:uniqueness of LU}}\label{proof:uniqueness of LU}

\textbf{Existence.} Since \(\det(X(1:j,1:j))\neq 0\) for all \(j\in[m]\), we may perform \(m\) steps of Gauss elimination with out pivoting on \(X\). Namely
\[
L_m\dots L_1X=\hat{U},
\]
where each \(L_i\in\mathbb{F}^{n\times n}\) is unit lower triangular and \(\hat{U}\in\mathbb{F}^{n\times m}\) is upper triangular. Multiplying the above equation on the left by \(\hat{L}:=L_1^{-1}\dots L_m^{-1}\), which is also unit lower triangular, gives
\[
X=\hat{L}\hat{U}.
\]
Since the last \(n-m\) rows of \(\hat{U}\) are all zero, we have
\[
X=\begin{bmatrix}
\hat{L}(:,1:m) & \hat{L}(:,m+1:n)
\end{bmatrix}\begin{bmatrix}
\hat{U}(1:m,:)\\
\hat{U}(m+1:n,:)
\end{bmatrix}=\hat{L}(:,1:m)\hat{U}(1:m,:).
\]
Here \(\hat{L}(:,1:m)\) is unit lower triangular and \(\hat{U}(1:m,:)\in\mathbb{F}^{m\times m}\) is upper triangular.

\textbf{Uniqueness.} Let \(L\in\mathbb{F}^{n\times m}\) be unit lower triangular and \(U\in\mathbb{F}^{m\times m}\) be upper triangular with \(X=LU\). Since \(\mathrm{rank}(X)=m\), the matrix \(U\) is invertible, hence all diagonal entries of \(U\) are nonzero. The equation \(X=LU\) gives
\begin{equation}\label{equation:X=LU}
X(:,j)=L(:,1)U(1,j)+\dots+L(:,j)U(j,j),\qquad\forall j\in[m].
\end{equation}
We shall prove the uniqueness of \(L\) and \(U\) by induction.

Induction basis. Since \(L(1,1)=1\), the first entry of \(X(:,1)=L(:,1)U(1,1)\) gives the uniqueness of \(U(1,1)\). Since \(U(1,1)\neq 0\), it also gives the uniqueness of \(L(:,1)=X(:,1)/U(1,1)\).

Inductive step. Suppose that the first \(j-1\) columns of \(L\) and of \(U\) are unique. The first entry of \eqref{equation:X=LU} gives the uniqueness of \(U(1,j)\), the second entry gives that of \(U(2,j)\), and so on, up to the \(j\)-th entry, with gives the uniqueness of \(U(j,j)\). In conclusion, the \(j\)-th column of \(U\) is unique. Moreover
\[
L(:,j)=(X(:,j)-L(:,1)U(1,j)-\dots-L(:,j-1)U(j-1,j))/U(j,j)
\]
gives the uniqueness of \(L(:,j)\).

By the principle of mathematical induction, \(L\) and \(U\) are both unique.

\section{Proof of Lemma \ref{lemma:eigen decomposition of T}}\label{proof:eigen decomposition of T}
\textbf{Existence of \(W\).} Since the \(n\) diagonal entries of \(T\) are distinct, \(T\) is diagonalizable on \(\mathbb{F}\). \(\forall i\in[n]\). Denote by \(V_i\) the eigenspace of \(T\) for \(\lambda_i\). Take \(v_i\in V_i\setminus\{0\}\) and put
\[
l_i:=\max\{j\mid v_i(j)\neq 0\}\in[n].
\]
Noting that
\[
(Tv_i)(l_i)=\sum_{j\in[n]}T(l_i,j)v_i(j)=T(l_i,l_i)v_i(l_i)=\lambda_{l_i}v_i(l_i),
\]
and
\[
(\lambda_iv_i)(l_i)=\lambda_iv_i(l_i).
\]
Then \(Tv_i=\lambda_iv_i\) gives \(\lambda_{l_i}v_i(l_i)=\lambda_iv_i(l_i)\). Since \(v_i(l_i)\neq 0\), we have \(\lambda_{l_i}=\lambda_i\). Since the \(\{\lambda_j\}_{j\in[n]}\) are pairwise distinct, we have \(l_i=i\).

Define \(W\in\mathbb{F}^{n\times n}\) by \(W(:,i):=v_i/v_i(i)\) for all \(i\in[n]\). The \(W\in\mathbb{F}^{n\times n}\) is unit upper triangular and \(TW=W\Lambda\).

\textbf{Uniqueness of \(W\).} Let \(W_1\), \(W_2\in\mathbb{F}^{n\times n}\) be unit upper triangular with \(T=W_1\Lambda W_1^{-1}\) and \(T=W_2\Lambda W_2^{-1}\). \(\forall i\in[n]\), the columns \(W_1(:,i)\) and \(W_2(:,i)\) are both eigenvectors of \(T\) for \(\lambda_i\), so
\[
W_1(:,i)=cW_2(:,i)
\]
for some \(c\in\mathbb{F}\setminus\{0\}\). Comparing the \(i\)-th coordinates gives \(c=1\). Hence \(W_1(:,i)=W_2(:,i)\). As this holds for every \(i\in[n]\), we have \(W_1=W_2\).

\section{Proof of Lemma \ref{lemma:key lemma}}\label{proof:key lemma}

\subsection{Definition of \(\Sigma\) and \(S\)}
Define \(C_1=\max_{j'\in[m]\setminus\{1\}}\left\|L(j':n,1:j'-1)L(1:j'-1,1:j'-1)^{-1}\right\|_{\infty}\).

Fix \(j\in[m]\setminus\{1\}\).

Define \(\Sigma:=L(j:n,1:j-1)L(1:j-1,1:j-1)^{-1}\in\mathbb{F}^{(n-j+1)\times(j-1)}\). We have \(\|\Sigma\|_{\infty}\leq C_1\) and \(|\Sigma(r,i)|\leq C_1\) for all \((r,i)\in[n-j+1]\times[j-1]\).

Define \(S:=L_k(j:n,1:j-1)L_k(1:j-1,1:j-1)^{-1}\in\mathbb{F}^{(n-j+1)\times(j-1)}\). Direct calculation yields
\begin{align*}
S&=L_k(j:n,1:j-1)L_k(1:j-1,1:j-1)^{-1}\\
&=\Lambda_{j-1}^{\perp k}L(j:n,1:j-1)\Lambda_{j-1}^{-k}(\Lambda_{j-1}^kL(1:j-1,1:j-1)\Lambda_{j-1}^{-k})^{-1}\\
&=\Lambda_{j-1}^{\perp k}L(j:n,1:j-1)\Lambda_{j-1}^{-k}\Lambda_{j-1}^kL(1:j-1,1:j-1)^{-1}\Lambda_{j-1}^{-k}\\
&=\Lambda_{j-1}^{\perp k}L(j:n,1:j-1)L(1:j-1,1:j-1)^{-1}\Lambda_{j-1}^{-k}\\
&=\Lambda_{j-1}^{\perp k}\Sigma\Lambda_{j-1}^{-k}.
\end{align*}
Therefore
\[
S(r,i)=\Sigma(r,i)\left(\dfrac{\lambda_{j-1+r}}{\lambda_i}\right)^k,\qquad\forall(r,i)\in[n-j+1]\times[j-1],
\]
and
\[
|S(r,i)|=|\Sigma(r,i)|\left|\dfrac{\lambda_{j-1+r}}{\lambda_i}\right|^k\leq C_1\left|\dfrac{\lambda_{j-1+r}}{\lambda_i}\right|^k,\qquad\forall(r,i)\in[n-j+1]\times[j-1].
\]

\subsection{Definition of \(S'\)}

Define \(W_{11}:=W(1:j-1,1:j-1)\), \(W_{12}:=W(1:j-1,j:n)\), \(W_{22}:=W(j:n,j:n)\).

Define \(H:=W_{11}^{-1}W_{12}S\in\mathbb{F}^{(j-1)\times(j-1)}\).

Define \(C_2:=\|W\|_{\infty}\|W^{-1}\|_{\infty}\), \(C_3:=2C_1C_2\).

There exists \(k_1\in\mathbb{N}_{+}\) such that for \(k\geq k_1\), \(\rho^k\leq 1/2\) and \(2C_3\rho^k\leq 1/2\). Suppose \(k\geq k_1\) from now on.

We have
\[
\|W_{11}^{-1}W_{12}\|_{\infty}\leq\|(W^{-1})_{11}\|_{\infty}\|W_{12}\|_{\infty}\leq\|W^{-1}\|_{\infty}\|W\|_{\infty}=C_2.
\]
Hence \(\forall(i,i')\in[j-1]^2\),
\[
|H(i,i')|\leq\sum_{s\in[n-j+1]}|(W_{11}^{-1}W_{12})(i,s)||S(s,i')|\leq\sum_{s\in[n-j+1]}C_2C_1\left|\dfrac{\lambda_{j-1+s}}{\lambda_{i'}}\right|^k\leq C_3\left|\dfrac{\lambda_j}{\lambda_{i'}}\right|^k.
\]
Therefore
\[
\|H\|_{\infty}=\max_{i\in[j-1]}\sum_{i'\in[j-1]}|H(i,i')|\leq C_3\sum_{i'\in[j-1]}\left|\dfrac{\lambda_j}{\lambda_{i'}}\right|^k\leq 2C_3\rho^k\leq 1/2.
\]
Namely \(W_{11}+W_{12}S=W_{11}(I_{j-1}+W_{11}^{-1}W_{12}S)=W_{11}(I_{j-1}+H)\) is invertible.

Define \(S':=W_{22}S(W_{11}+W_{12}S)^{-1}\in\mathbb{F}^{(n-j+1)\times(j-1)}\).

\subsection{Definition of \(\Sigma'\)}

Define \(W_k:=\Lambda^{-k}W\Lambda^k\). Define \(W_{k,11}=W_k(1:j-1,1:j-1)\), \(W_{k,12}=W_k(1:j-1,j:n)\), \(W_{k,22}=W_k(j:n,j:n)\). Direct calculation yields
\begin{align*}
\|W_{k,11}-I_{j-1}\|_{\infty}&=\max_{i\in[j-1]}\sum_{i'\in[j-1]}|W_{k,11}(i,i')-\delta_{i,i'}|=\max_{i\in[j-1]}\sum_{i'\in[j-1]\setminus[i]}|W_{k,11}(i,i')-\delta_{i,i'}|\\
&=\max_{i\in[j-1]}\sum_{i'\in[j-1]\setminus[i]}|W_{k,11}(i,i')|=\max_{i\in[j-1]}\sum_{i'\in[j-1]\setminus[i]}\left|W_{11}(i,i')\dfrac{\lambda_{i'}^k}{\lambda_i^k}\right|\\
&\leq\max_{i\in[j-1]}\sum_{i'\in[j-1]\setminus[i]}C_2\left|\dfrac{\lambda_{i'}}{\lambda_i}\right|^k\leq 2C_2\rho^k.
\end{align*}
\begin{align*}
\|W_{k,22}-I_{n-j+1}\|_{\infty}&=\max_{i\in[n-j+1]}\sum_{i'\in[n-j+1]}|W_{k,22}(i,i')-\delta_{i,i'}|=\max_{i\in[n-j+1]}\sum_{i'\in[n-j+1]\setminus[i]}|W_{k,22}(i,i')-\delta_{i,i'}|\\
&=\max_{i\in[n-j+1]}\sum_{i'\in[n-j+1]\setminus[i]}|W_{k,22}(i,i')|=\max_{i\in[n-j+1]}\sum_{i'\in[n-j+1]\setminus[i]}\left|W_{22}(i,i')\dfrac{\lambda_{j-1+i'}^k}{\lambda_{j-1+i}^k}\right|\\
&\leq\max_{i\in[n-j+1]}\sum_{i'\in[n-j+1]\setminus[i]}C_2\left|\dfrac{\lambda_{j-1+i'}}{\lambda_{j-1+i}}\right|^k\leq 2C_2\rho^k.
\end{align*}
\begin{align*}
\|W_{k,12}\|_{\infty}&=\max_{i\in[j-1]}\sum_{i'\in[n-j+1]}|W_{k,12}(i,i')|=\max_{i\in[j-1]}\sum_{i'\in[n-j+1]}\left|W_{12}(i,i')\dfrac{\lambda_{j-1+i'}^k}{\lambda_i^k}\right|\\
&\leq\max_{i\in[j-1]}\sum_{i'\in[n-j+1]}C_2\left|\dfrac{\lambda_{j-1+i'}}{\lambda_i}\right|^k\leq 2C_2\rho^k.
\end{align*}
Define \(C_4:=2C_2(1+C_1)\).

There exists \(k_2\geq k_1\) such that for \(k\geq k_2\), \(C_4\rho^k\leq 1/8\). Suppose \(k\geq k_2\) from now on.

Then we have
\[
\|W_{k,11}+W_{k,12}\Sigma-I_{j-1}\|_{\infty}\leq\|W_{k,11}-I_{j-1}\|_{\infty}+\|W_{k,12}\|_{\infty}\|\Sigma\|_{\infty}\leq 2C_2\rho^k+2C_2\rho^kC_1=C_4\rho^k\leq 1/8.
\]
Therefore \(W_{k,11}+W_{k,12}\Sigma=I_{j-1}+(W_{k,11}+W_{k,12}\Sigma-I_{j-1})\) is invertible. Furthermore,
\begin{align*}
\|(W_{k,11}+W_{k,12}\Sigma)^{-1}\|_{\infty}&=\|(I_{j-1}+(W_{k,11}+W_{k,12}\Sigma-I_{j-1}))^{-1}\|_{\infty}\\
&\leq 1/(1-\|W_{k,11}+W_{k,12}\Sigma-I_{j-1}\|_{\infty})\\
&\leq\dfrac{8}{7},
\end{align*}
\[
\|(W_{k,11}+W_{k,12}\Sigma)^{-1}-I_{j-1}\|_{\infty}\leq\dfrac{8}{7}C_4\rho^k.
\]
Define \(\Sigma':=W_{k,22}\Sigma(W_{k,11}+W_{k,12}\Sigma)^{-1}\in\mathbb{F}^{(n-j+1)\times(j-1)}\).

Define \(C_5:=C_1(4C_2+2C_4)\).

Then
\begin{align*}
\|\Sigma'-\Sigma\|_{\infty}&=\|W_{k,22}\Sigma(W_{k,11}+W_{k,12}\Sigma)^{-1}-\Sigma\|_{\infty}\\
&=\|(W_{k,22}-I_{n-j+1})\Sigma(W_{k,11}+W_{k,12}\Sigma)^{-1}+\Sigma((W_{k,11}+W_{k,12}\Sigma)^{-1}-I_{j-1})\|_{\infty}\\
&\leq\|W_{k,22}-I_{n-j+1}\|_{\infty}\|\Sigma\|_{\infty}\|(W_{k,11}+W_{k,12}\Sigma)^{-1}\|_{\infty}+\|\Sigma\|_{\infty}\|(W_{k,11}+W_{k,12}\Sigma)^{-1}-I_{j-1}\|_{\infty}\\
&\leq 2C_2\rho^kC_1\dfrac{8}{7}+C_1\dfrac{8}{7}C_4\rho^k\\
&\leq 4C_1C_2\rho^k+2C_1C_4\rho^k\\
&=C_5\rho^k.
\end{align*}
Hence
\[
|\Sigma'(r,i)-\Sigma(r,i)|\leq\|\Sigma'-\Sigma\|_{\infty}\leq C_5\rho^k,\qquad\forall(r,i)\in[n-j+1]\times[j-1].
\]
And \(\forall(r,i)\in[n-j+1]\times[j-1]\),
\begin{align*}
|\Sigma'(r,i)|&\leq|\Sigma'(r,i)-\Sigma(r,i)|+|\Sigma(r,i)|\leq C_5\rho^k+C_1=C_1(4C_2+2C_4)\rho^k+C_1\\
&=C_1\left(4\dfrac{C_4}{2(1+C_1)}+2C_4\right)\rho^k+C_1\leq C_14C_4\rho^k+C_1\leq C_1\dfrac{1}{2}+C_1\leq 2C_1.
\end{align*}
Then
\begin{align*}
S'&=W_{22}S(W_{11}+W_{12}S)^{-1}\\
&=\Lambda_{j-1}^{\perp k}W_{k,22}\Lambda_{j-1}^{\perp -k}\Lambda_{j-1}^{\perp k}\Sigma\Lambda_{j-1}^{-k}(\Lambda_{j-1}^kW_{k,11}\Lambda_{j-1}^{-k}+\Lambda_{j-1}^kW_{k,12}\Lambda_{j-1}^{\perp -k}\Lambda_{j-1}^{\perp k}\Sigma\Lambda_{j-1}^{-k})^{-1}\\
&=\Lambda_{j-1}^{\perp k}W_{k,22}\Sigma(W_{k,11}+W_{k,12}\Sigma)^{-1}\Lambda_{j-1}^{-k}=\Lambda_{j-1}^{\perp k}\Sigma'\Lambda_{j-1}^{-k}.
\end{align*}
Therefore
\[
S'(r,i)=\Sigma'(r,i)\left(\dfrac{\lambda_{j-1+r}}{\lambda_i}\right)^k,\qquad\forall(r,i)\in[n-j+1]\times[j-1],
\]
and
\[
|S'(r,i)|\leq|\Sigma'(r,i)|\left|\dfrac{\lambda_{j-1+r}}{\lambda_i}\right|^k\leq 2C_1\left|\dfrac{\lambda_{j-1+r}}{\lambda_i}\right|^k,\qquad\forall(r,i)\in[n-j+1]\times[j-1].
\]

\subsection{Bound of \(SS^*\) and \(S'S'^*\)}

There exists \(k_3\geq k_2\) such that for \(k\geq k_3\), \(16C_1^2\rho^{2k}\leq 1/2\). Suppose \(k\geq k_3\) from now on.

\(\forall(r,r')\in[n-j+1]^2\), direct calculation yields
\begin{align*}
|(SS^*)(r,r')|&=\sum_{i\in[j-1]}|S(r,i)||S^*(i,r')|=\sum_{i\in[j-1]}|S(r,i)||S(r',i)|\\
&\leq C_1^2\sum_{i\in[j-1]}\left|\dfrac{\lambda_{j-1+r}}{\lambda_i}\right|^k\left|\dfrac{\lambda_{j-1+r'}}{\lambda_i}\right|^k\\
&=C_1^2\left|\dfrac{\lambda_{j-1+r}}{\lambda_j}\right|^k\left|\dfrac{\lambda_{j-1+r'}}{\lambda_j}\right|^k\sum_{i\in[j-1]}\left|\dfrac{\lambda_j}{\lambda_i}\right|^{2k}\\
&\leq 2C_1^2\left|\dfrac{\lambda_{j-1+r}}{\lambda_j}\right|^k\left|\dfrac{\lambda_{j-1+r'}}{\lambda_j}\right|^k\rho^{2k}.
\end{align*}
And then
\begin{align*}
\|SS^*\|_{\infty}&=\max_{r\in[n-j+1]}\sum_{r'\in[n-j+1]}|(SS^*)(r,r')|\\
&\leq\max_{r\in[n-j+1]}\sum_{r'\in[n-j+1]}2C_1^2\left|\dfrac{\lambda_{j-1+r}}{\lambda_j}\right|^k\left|\dfrac{\lambda_{j-1+r'}}{\lambda_j}\right|^k\rho^{2k}\\
&\leq\max_{r\in[n-j+1]}4C_1^2\left|\dfrac{\lambda_{j-1+r}}{\lambda_j}\right|^k\rho^{2k}\\
&\leq 4C_1^2\rho^{2k}\leq 1/2.
\end{align*}
Similarly, \(\forall(r,r')\in[n-j+1]^2\), direct calculation yields
\begin{align*}
|(S'S'^*)(r,r')|&=\sum_{i\in[j-1]}|S'(r,i)||S'^*(i,r')|=\sum_{i\in[j-1]}|S'(r,i)||S'(r',i)|\\
&\leq 4C_1^2\sum_{i\in[j-1]}\left|\dfrac{\lambda_{j-1+r}}{\lambda_i}\right|^k\left|\dfrac{\lambda_{j-1+r'}}{\lambda_i}\right|^k\\
&=4C_1^2\left|\dfrac{\lambda_{j-1+r}}{\lambda_j}\right|^k\left|\dfrac{\lambda_{j-1+r'}}{\lambda_j}\right|^k\sum_{i\in[j-1]}\left|\dfrac{\lambda_j}{\lambda_i}\right|^{2k}\\
&\leq 8C_1^2\left|\dfrac{\lambda_{j-1+r}}{\lambda_j}\right|^k\left|\dfrac{\lambda_{j-1+r'}}{\lambda_j}\right|^k\rho^{2k}.
\end{align*}
And then
\begin{align*}
\|S'S'^*\|_{\infty}&=\max_{r\in[n-j+1]}\sum_{r'\in[n-j+1]}|(S'S'^*)(r,r')|\\
&\leq\max_{r\in[n-j+1]}\sum_{r'\in[n-j+1]}8C_1^2\left|\dfrac{\lambda_{j-1+r}}{\lambda_j}\right|^k\left|\dfrac{\lambda_{j-1+r'}}{\lambda_j}\right|^k\rho^{2k}\\
&\leq\max_{r\in[n-j+1]}16C_1^2\left|\dfrac{\lambda_{j-1+r}}{\lambda_j}\right|^k\rho^{2k}\\
&\leq 16C_1^2\rho^{2k}\leq 1/2.
\end{align*}

\subsection{The orthogonal projection onto a graph}\label{subsection:the orthogonal projection onto a graph}

\begin{proposition}
Let \(\Theta\in\mathbb{F}^{(n-j+1)\times(j-1)}\) and \(\Gamma:=\mathrm{graph}(\Theta)=\mathrm{span}\begin{pmatrix}
I_{j-1}\\
\Theta
\end{pmatrix}\). Then
\[
P_\Gamma=\begin{pmatrix}
I_{j-1}\\
\Theta
\end{pmatrix}(I_{j-1}+\Theta^*\Theta)^{-1}\begin{pmatrix}
I_{j-1} & \Theta^*
\end{pmatrix},
\]
and, for every \(z\in\mathbb{F}^n\),
\[
(I_n-P_\Gamma)z=\begin{pmatrix}
\Theta^*(I_{n-j+1}+\Theta\Theta^*)^{-1}(\Theta z(1:j-1)-z(j:n))\\
z(j:n)-\Theta(I_{j-1}+\Theta^*\Theta)^{-1}(z(1:j-1)+\Theta^*z(j:n))
\end{pmatrix}.
\]
\end{proposition}
\begin{proof}
Since \(\begin{pmatrix}
I_{j-1}\\
\Theta
\end{pmatrix}\) has full column rank, an orthonormal basis of \(\Gamma\) is
\[
\begin{pmatrix}
I_{j-1}\\
\Theta
\end{pmatrix}\left(\begin{pmatrix}
I_{j-1} & \Theta^*
\end{pmatrix}\begin{pmatrix}
I_{j-1}\\
\Theta
\end{pmatrix}\right)^{-1/2}=\begin{pmatrix}
I_{j-1}\\
\Theta
\end{pmatrix}\left(I_{j-1}+\Theta^*\Theta\right)^{-1/2}.
\]
Therefore
\begin{align*}
P_\Gamma&=\begin{pmatrix}
I_{j-1}\\
\Theta
\end{pmatrix}\left(I_{j-1}+\Theta^*\Theta\right)^{-1/2}\left(\begin{pmatrix}
I_{j-1}\\
\Theta
\end{pmatrix}\left(I_{j-1}+\Theta^*\Theta\right)^{-1/2}\right)^*\\
&=\begin{pmatrix}
I_{j-1}\\
\Theta
\end{pmatrix}\left(I_{j-1}+\Theta^*\Theta\right)^{-1}\begin{pmatrix}
I_{j-1} & \Theta^*
\end{pmatrix}.
\end{align*}
\(\forall z\in\mathbb{F}^n\). Denote \(z_1:=z(1:j-1)\in\mathbb{F}^{j-1}\) and \(z_2:=z(j:n)\in\mathbb{F}^{n-j+1}\). Direct calculation yields
\begin{align*}
(I_n-P_\Gamma)z&=\begin{pmatrix}
z_1\\
z_2
\end{pmatrix}-\begin{pmatrix}
I_{j-1}\\
\Theta
\end{pmatrix}\left(I_{j-1}+\Theta^*\Theta\right)^{-1}\begin{pmatrix}
I_{j-1} & \Theta^*
\end{pmatrix}\begin{pmatrix}
z_1\\
z_2
\end{pmatrix}\\
&=\begin{pmatrix}
z_1\\
z_2
\end{pmatrix}-\begin{pmatrix}
I_{j-1}\\
\Theta
\end{pmatrix}\left(I_{j-1}+\Theta^*\Theta\right)^{-1}(z_1+\Theta^*z_2)\\
&=\begin{pmatrix}
z_1\\
z_2
\end{pmatrix}-\begin{pmatrix}
\left(I_{j-1}+\Theta^*\Theta\right)^{-1}(z_1+\Theta^*z_2)\\
\Theta\left(I_{j-1}+\Theta^*\Theta\right)^{-1}(z_1+\Theta^*z_2)
\end{pmatrix}\\
&=\begin{pmatrix}
z_1-\left(I_{j-1}+\Theta^*\Theta\right)^{-1}(z_1+\Theta^*z_2)\\
z_2-\Theta\left(I_{j-1}+\Theta^*\Theta\right)^{-1}(z_1+\Theta^*z_2)
\end{pmatrix}.
\end{align*}
Note the equation
\[
\Theta^*(I_{n-j+1}+\Theta\Theta^*)=(I_{j-1}+\Theta^*\Theta)\Theta^*,
\]
left multiply by \((I_{j-1}+\Theta^*\Theta)^{-1}\) and right multiply by \((I_{n-j+1}+\Theta\Theta^*)^{-1}\) gives
\[
(I_{j-1}+\Theta^*\Theta)^{-1}\Theta^*=\Theta^*(I_{n-j+1}+\Theta\Theta^*)^{-1}.
\]
Finally
\begin{align*}
(I_n-P_\Gamma)z&=\begin{pmatrix}
z_1-\left(I_{j-1}+\Theta^*\Theta\right)^{-1}(z_1+\Theta^*z_2)\\
z_2-\Theta\left(I_{j-1}+\Theta^*\Theta\right)^{-1}(z_1+\Theta^*z_2)
\end{pmatrix}\\
&=\begin{pmatrix}
(I_{j-1}+\Theta^*\Theta)^{-1}\Theta^*(\Theta z_1-z_2)\\
z_2-\Theta\left(I_{j-1}+\Theta^*\Theta\right)^{-1}(z_1+\Theta^*z_2)
\end{pmatrix}\\
&=\begin{pmatrix}
\Theta^*(I_{n-j+1}+\Theta\Theta^*)^{-1}(\Theta z_1-z_2)\\
z_2-\Theta\left(I_{j-1}+\Theta^*\Theta\right)^{-1}(z_1+\Theta^*z_2)
\end{pmatrix}.
\end{align*}
Which completes the proof.
\end{proof}

\subsection{Definition of \(\varphi\) and \(v\)}

Denote \(G:=\mathrm{span}(L_k(:,1:j-1))\). We have
\begin{align*}
G&=\mathrm{span}\begin{pmatrix}
L_k(1:j-1,1:j-1)\\
L_k(j:n,1:j-1)
\end{pmatrix}\\
&=\mathrm{span}\left(\begin{pmatrix}
L_k(1:j-1,1:j-1)\\
L_k(j:n,1:j-1)
\end{pmatrix}L_k(1:j-1,1:j-1)^{-1}\right)\\
&=\mathrm{span}\begin{pmatrix}
I_{j-1}\\
S
\end{pmatrix}=\mathrm{graph}(S).
\end{align*}
Denote \(F:=\mathrm{span}((WL_k)(:,1:j-1))\). We have
\begin{align*}
F&=\mathrm{span}((WL_k)(:,1:j-1))\\
&=\mathrm{span}(WL_k(:,1:j-1))\\
&=\mathrm{span}\left(\begin{pmatrix}
W_{11} & W_{12} \\
0 & W_{22}
\end{pmatrix}\begin{pmatrix}
L_k(1:j-1,1:j-1)\\
L_k(j:n,1:j-1)
\end{pmatrix}\right)\\
&=\mathrm{span}\left(\begin{pmatrix}
W_{11} & W_{12} \\
0 & W_{22}
\end{pmatrix}\begin{pmatrix}
L_k(1:j-1,1:j-1)\\
L_k(j:n,1:j-1)
\end{pmatrix}L_k(1:j-1,1:j-1)^{-1}\right)\\
&=\mathrm{span}\left(\begin{pmatrix}
W_{11} & W_{12} \\
0 & W_{22}
\end{pmatrix}\begin{pmatrix}
I_{j-1}\\
S
\end{pmatrix}\right)\\
&=\mathrm{span}\left(\begin{pmatrix}
W_{11}+W_{12}S\\
W_{22}S
\end{pmatrix}\right)\\
&=\mathrm{span}\left(\begin{pmatrix}
W_{11}+W_{12}S\\
W_{22}S
\end{pmatrix}(W_{11}+W_{12}S)^{-1}\right)\\
&=\mathrm{span}\left(\begin{pmatrix}
I_{j-1}\\
S'
\end{pmatrix}\right)=\mathrm{graph}(S').
\end{align*}
Define
\[
\varphi:=(I_n-P_G)L_k(:,j),\qquad v:=(I_n-P_F)(WL_k)(:,j).
\]
By the definition of Gram--Schmidt orthogonalization, we know that
\[
Q(L_k)(:,j)=\varphi/\|\varphi\|_2,\qquad Q(WL_k)(:,j):=v/\|v\|_2.
\]

\subsection{Explicit formula of \(\varphi\) and \(v\)}

Define \(b:=L_k(j:n,j)\in\mathbb{F}^{n-j+1}\). By subsection \ref{subsection:the orthogonal projection onto a graph},
\begin{align*}
\varphi&=(I_n-P_G)L_k(:,j)\\
&=\begin{pmatrix}
-S^*(I_{n-j+1}+SS^*)^{-1}b\\
(I_{n-j+1}-S(I_{j-1}+S^*S)^{-1}S^*)b
\end{pmatrix}=\begin{pmatrix}
-S^*(I_{n-j+1}+SS^*)^{-1}b\\
(I_{n-j+1}+SS^*)^{-1}b
\end{pmatrix}.
\end{align*}
And
\begin{align*}
v&=(I_n-P_F)(WL_k)(:,j)=(I_n-P_F)\begin{pmatrix}
W_{12}b\\
W_{22}b
\end{pmatrix}\\
&=\begin{pmatrix}
S'^*(I_{n-j+1}+S'S'^*)^{-1}(S'W_{12}b-W_{22}b)\\
W_{22}b-S'(I_{j-1}+S'^*S')^{-1}(W_{12}b+S'^*W_{22}b)
\end{pmatrix}.
\end{align*}

\subsection{The first \(j-1\) rows of \(\varphi\) and \(v\)}

Define \(C_6:=\max\{|L(i,j)|\mid (i,j)\in[n]\times[m]\}\). Therefore \(\forall r\in[n-j+1]\),
\[
|b(r)|=|L_k(j-1+r,j)|=|L(j-1+r,j)|\left|\dfrac{\lambda_{j-1+r}}{\lambda_j}\right|^k\leq C_6\left|\dfrac{\lambda_{j-1+r}}{\lambda_j}\right|^k.
\]
And as a corollary, \(\|b\|_{\infty}\leq C_6\).

\subsubsection{Bound of \(|\varphi(i)+(S^*b)_i|\)}

Note that \((I_{n-j+1}+SS^*)(I_{n-j+1}+SS^*)^{-1}=I_{n-j+1}\), we have
\[
(I_{n-j+1}+SS^*)^{-1}=I_{n-j+1}-SS^*(I_{n-j+1}+SS^*)^{-1}.
\]
\(\forall i\in[j-1]\),
\begin{align*}
\varphi(i)&=-(S^*(I_{n-j+1}+SS^*)^{-1}b)_i\\
&=-(S^*b)_i+(S^*SS^*(I_{n-j+1}+SS^*)^{-1}b)_i.
\end{align*}
Define \(C_7:=16C_1^3C_6\).

Direct calculation yields
\begin{align*}
|\varphi(i)+(S^*b)_i|&=|(S^*SS^*(I_{n-j+1}+SS^*)^{-1}b)_i|\\
&\leq\sum_{r\in[n-j+1]}|S^*(i,r)||(SS^*(I_{n-j+1}+SS^*)^{-1}b)_r|\\
&\leq\sum_{r\in[n-j+1]}|S(r,i)|\left(\sum_{s\in[n-j+1]}|(SS^*)(r,s)||((I_{n-j+1}+SS^*)^{-1}b)_s|\right)\\
&\leq2C_6\sum_{r\in[n-j+1]}|S(r,i)|\left(\sum_{s\in[n-j+1]}|(SS^*)(r,s)|\right)\\
&\leq8C_1^2C_6\rho^{2k}\sum_{r\in[n-j+1]}|S(r,i)|\left|\dfrac{\lambda_{j-1+r}}{\lambda_j}\right|^k\\
&\leq8C_1^3C_6\rho^{2k}\sum_{r\in[n-j+1]}\left|\dfrac{\lambda_{j-1+r}}{\lambda_i}\right|^k\left|\dfrac{\lambda_{j-1+r}}{\lambda_j}\right|^k\\
&=8C_1^3C_6\rho^{2k}\left|\dfrac{\lambda_j}{\lambda_i}\right|^k\sum_{r\in[n-j+1]}\left|\dfrac{\lambda_{j-1+r}}{\lambda_j}\right|^{2k}\\
&\leq16C_1^3C_6\rho^{2k}\left|\dfrac{\lambda_j}{\lambda_i}\right|^k\\
&\leq C_7\rho^k\left|\dfrac{\lambda_j}{\lambda_i}\right|^k.
\end{align*}

\subsubsection{Bound of \(|v(i)+(S'^*W_{22}b)_i|\)}

By the equation
\[
(I_{n-j+1}+S'S'^*)^{-1}=I_{n-j+1}-S'S'^*(I_{n-j+1}+S'S'^*)^{-1},
\]
\(\forall i\in[j-1]\), we have
\begin{align*}
v(i)&=(S'^*(I_{n-j+1}+S'S'^*)^{-1}(S'W_{12}b-W_{22}b))_i\\
&=((S'^*-S'^*S'S'^*(I_{n-j+1}+S'S'^*)^{-1})(S'W_{12}b-W_{22}b))_i\\
&=(S'^*S'W_{12}b-S'^*W_{22}b-S'^*S'S'^*(I_{n-j+1}+S'S'^*)^{-1}(S'W_{12}b-W_{22}b))_i\\
&=-(S'^*W_{22}b)_i+(S'^*S'W_{12}b)_i-(S'^*S'S'^*(I_{n-j+1}+S'S'^*)^{-1}(S'W_{12}b-W_{22}b))_i.
\end{align*}
Define \(C_8:=(16+64C_1(4C_1+1))C_1^2C_2C_6\). Therefore we have
\begin{align*}
|v(i)+(S'^*W_{22}b)_i|&\leq|(S'^*S'W_{12}b)_i|+|(S'^*S'S'^*(I_{n-j+1}+S'S'^*)^{-1}(S'W_{12}b-W_{22}b))_i|\\
&\leq\sum_{r\in[n-j+1]}|S'^*(i,r)||(S'W_{12}b)_r|\\
&+\sum_{r\in[n-j+1]}|(S'^*S'S'^*)(i,r)||((I_{n-j+1}+S'S'^*)^{-1}(S'W_{12}b-W_{22}b))_r|\\
&\leq\sum_{r\in[n-j+1]}|S'(r,i)|\sum_{i'\in[j-1]}|S'(r,i')||(W_{12}b)_{i'}|\\
&+2\sum_{r\in[n-j+1]}|(S'^*S'S'^*)(i,r)|(\|S'W_{12}b\|_{\infty}+\|W_{22}b\|_{\infty})\\
&\leq C_2C_6\sum_{r\in[n-j+1]}|S'(r,i)|\sum_{i'\in[j-1]}|S'(r,i')|\\
&+2(4C_1C_2C_6+C_2C_6)\sum_{r\in[n-j+1]}\sum_{r'\in[n-j+1]}|S'(r',i)||(S'S'^*)(r',r)|\\
&\leq C_2C_6\sum_{r\in[n-j+1]}2C_1\left|\dfrac{\lambda_{j-1+r}}{\lambda_i}\right|^k\sum_{i'\in[j-1]}2C_1\left|\dfrac{\lambda_{j-1+r}}{\lambda_{i'}}\right|^k\\
&+2(4C_1C_2C_6+C_2C_6)\sum_{r\in[n-j+1]}\sum_{r'\in[n-j+1]}2C_1\left|\dfrac{\lambda_{j-1+r'}}{\lambda_i}\right|^k8C_1^2\left|\dfrac{\lambda_{j-1+r}}{\lambda_j}\right|^k\left|\dfrac{\lambda_{j-1+r'}}{\lambda_j}\right|^k\rho^{2k}\\
&\leq 16C_1^2C_2C_6\left|\dfrac{\lambda_j}{\lambda_i}\right|^k\rho^k\\
&+128C_1^3(4C_1C_2C_6+C_2C_6)\rho^{2k}\left|\dfrac{\lambda_j}{\lambda_i}\right|^k\\
&\leq \left(16+64C_1(4C_1+1)\right)C_1^2C_2C_6\rho^{k}\left|\dfrac{\lambda_j}{\lambda_i}\right|^k=C_8\rho^k\left|\dfrac{\lambda_j}{\lambda_i}\right|^k.
\end{align*}

\subsubsection{Definition and bound of \(\gamma\)}

\(\forall r\in[n-j+1]\), define \(\gamma(r):=(W_{22}b)(r)\left(\dfrac{\lambda_j}{\lambda_{j-1+r}}\right)^k\).

Direct calculation yields
\begin{align*}
|\gamma(1)-1|&=|(W_{22}b)(1)-1|=\left|\sum_{r\in[n-j+1]}W_{22}(1,r)b(r)-1\right|\\
&=\left|\sum_{r\in[n-j+1]\setminus\{1\}}W_{22}(1,r)b(r)\right|\leq\sum_{r\in[n-j+1]\setminus\{1\}}C_2C_6\left|\dfrac{\lambda_{j-1+r}}{\lambda_j}\right|^k\leq 2C_2C_6\rho^k.
\end{align*}
And \(\forall r\in[n-j+1]\setminus\{1\}\),
\begin{align*}
|\gamma(r)|&\leq\left|\dfrac{\lambda_j}{\lambda_{j-1+r}}\right|^k|(W_{22}b)(r)|\leq\left|\dfrac{\lambda_j}{\lambda_{j-1+r}}\right|^k\sum_{r'\in[n-j+1]\setminus[r-1]}|W_{22}(r,r')||b(r')|\\
&\leq C_2C_6\sum_{r'\in[n-j+1]\setminus[r-1]}\left|\dfrac{\lambda_{j-1+r'}}{\lambda_{j-1+r}}\right|^k\leq 2C_2C_6.
\end{align*}

\subsubsection{Bound of \(|(S'^*W_{22}b)_i-(S^*b)_i|\)}

\(\forall i\in[j-1]\). Direct calculation yields
\begin{align*}
&(S'^*W_{22}b)_i-(S^*b)_i\\
&=\sum_{r\in[n-j+1]}S'^*(i,r)(W_{22}b)(r)-\sum_{r\in[n-j+1]}S^*(i,r)b(r)\\
&=\sum_{r\in[n-j+1]}\left(\overline{\Sigma'(r,i)\left(\dfrac{\lambda_{j-1+r}}{\lambda_i}\right)^k}\gamma(r)\left(\dfrac{\lambda_{j-1+r}}{\lambda_j}\right)^k-\overline{\Sigma(r,i)\left(\dfrac{\lambda_{j-1+r}}{\lambda_i}\right)^k}L(j-1+r,j)\left(\dfrac{\lambda_{j-1+r}}{\lambda_j}\right)^k\right)\\
&=\left(\dfrac{\lambda_i}{\lambda_j}\right)^k\sum_{r\in[n-j+1]}\left|\dfrac{\lambda_{j-1+r}}{\lambda_i}\right|^{2k}\left(\overline{\Sigma'(r,i)}\gamma(r)-\overline{\Sigma(r,i)}L(j-1+r,j)\right).
\end{align*}
Define \(C_9:=(8C_1C_2C_6+C_5+C_1C_6)\). Then we have
\begin{align*}
&|(S'^*W_{22}b)_i-(S^*b)_i|\\
&\leq\left|\dfrac{\lambda_i}{\lambda_j}\right|^k\sum_{r\in[n-j+1]}\left|\dfrac{\lambda_{j-1+r}}{\lambda_i}\right|^{2k}\left|\overline{\Sigma'(r,i)}\gamma(r)-\overline{\Sigma(r,i)}L(j-1+r,j)\right|\\
&\leq\left|\dfrac{\lambda_j}{\lambda_i}\right|^k\left|\overline{\Sigma'(1,i)}\gamma(1)-\overline{\Sigma(1,i)}\right|\\
&+\left|\dfrac{\lambda_i}{\lambda_j}\right|^k\sum_{r\in[n-j+1]\setminus\{1\}}\left|\dfrac{\lambda_{j-1+r}}{\lambda_i}\right|^{2k}\left|\overline{\Sigma'(r,i)}\gamma(r)-\overline{\Sigma(r,i)}L(j-1+r,j)\right|\\
&\leq\left|\dfrac{\lambda_j}{\lambda_i}\right|^k\left(\left|\overline{\Sigma'(1,i)}(\gamma(1)-1)\right|+\left|\overline{\Sigma'(1,i)}-\overline{\Sigma(1,i)}\right|\right)\\
&+\left|\dfrac{\lambda_i}{\lambda_j}\right|^k\sum_{r\in[n-j+1]\setminus\{1\}}\left|\dfrac{\lambda_{j-1+r}}{\lambda_i}\right|^{2k}\left(2C_1|\gamma(r)|+C_1C_6\right)\\
&\leq\left|\dfrac{\lambda_j}{\lambda_i}\right|^k\left(4C_1C_2C_6\rho^k+C_5\rho^k\right)\\
&+\left|\dfrac{\lambda_j}{\lambda_i}\right|^k\sum_{r\in[n-j+1]\setminus\{1\}}\left|\dfrac{\lambda_{j-1+r}}{\lambda_j}\right|^{2k}\left(4C_1C_2C_6+C_1C_6\right)\\
&\leq(4C_1C_2C_6+C_5)\left|\dfrac{\lambda_j}{\lambda_i}\right|^k\rho^k+(4C_1C_2C_6+C_1C_6)\left|\dfrac{\lambda_j}{\lambda_i}\right|^k\rho^k\\
&=(8C_1C_2C_6+C_5+C_1C_6)\left|\dfrac{\lambda_j}{\lambda_i}\right|^k\rho^k=C_9\left|\dfrac{\lambda_j}{\lambda_i}\right|^k\rho^k.
\end{align*}

\subsubsection{Bound of \(|v(i)-\varphi(i)|\)}

Define \(C_{10}:=C_7+C_8+C_9\).

\(\forall i\in[j-1]\), we have
\begin{align*}
|v(i)-\varphi(i)|&=|v(i)+(S'^*W_{22}b)_i-(S'^*W_{22}b)_i+(S^*b)_i-(S^*b)_i-\varphi(i)|\\
&\leq|(v(i)+(S'^*W_{22}b)_i)-((S'^*W_{22}b)_i-(S^*b)_i)-((S^*b)_i+\varphi(i))|\\
&\leq|v(i)+(S'^*W_{22}b)_i|+|(S'^*W_{22}b)_i-(S^*b)_i|+|(S^*b)_i+\varphi(i)|\\
&\leq C_8\rho^k\left|\dfrac{\lambda_j}{\lambda_i}\right|^k+C_9\rho^k\left|\dfrac{\lambda_j}{\lambda_i}\right|^k+C_7\rho^k\left|\dfrac{\lambda_j}{\lambda_i}\right|^k\\
&=C_{10}\rho^k\left|\dfrac{\lambda_j}{\lambda_i}\right|^k.
\end{align*}

\subsection{The last \(n-j+1\) rows of \(\varphi\) and \(v\)}

\subsubsection{Bound of \(\|S'\|_{\infty}\) and \(\|S'^*\|_{\infty}\)}

Direct calculation yields
\[
\|S'\|_{\infty}=\max_{r\in[n-j+1]}\sum_{i\in[j-1]}|S'(r,i)|\leq2C_1\sum_{i\in[j-1]}\left|\dfrac{\lambda_j}{\lambda_i}\right|^k\leq 4C_1\rho^k.
\]
\[
\|S'^*\|_{\infty}=\max_{i\in[j-1]}\sum_{r\in[n-j+1]}|S'(r,i)|\leq2C_1\sum_{r\in[n-j+1]}\left|\dfrac{\lambda_{j-1+r}}{\lambda_{j-1}}\right|^k\leq 4C_1\rho^k.
\]
Hence \(\|S'^*S'\|_{\infty}\leq\|S'^*\|_{\infty}\|S'\|_{\infty}\leq 16C_1^2\rho^{2k}\leq 1/2\).

\subsubsection{Bound of \(|v(j-1+r)-b(r)|\)}

Define \(C_{11}:=2C_2C_6+8C_1(C_2C_6+2C_1C_2C_6)\).

\(\forall r\in[n-j+1]\), direct calculation yields
\begin{align*}
&|v(j-1+r)-b(r)|\\
&=|(W_{22}b)_r-b(r)-(S'(I_{j-1}+S'^*S')^{-1}(W_{12}b+S'^*W_{22}b))_r|\\
&\leq|(W_{22}b)_r-b(r)|+|(S'(I_{j-1}+S'^*S')^{-1}(W_{12}b+S'^*W_{22}b))_r|\\
&\leq\sum_{s\in[n-j+1]\setminus[r]}|W_{22}(r,s)||b(s)|+\sum_{s\in[j-1]}|S'(r,s)|\|(I_{j-1}+S'^*S')^{-1}\|_{\infty}\|(W_{12}b+S'^*W_{22}b)\|_{\infty}\\
&\leq C_2C_6\sum_{s\in[n-j+1]\setminus[r]}\left|\dfrac{\lambda_{j-1+s}}{\lambda_j}\right|^k+4C_1\sum_{s\in[j-1]}\left|\dfrac{\lambda_{j-1+r}}{\lambda_s}\right|^k(C_2C_6+2C_1C_2C_6)\\
&\leq C_2C_6\left|\dfrac{\lambda_{j-1+r}}{\lambda_j}\right|^k2\rho^k+4C_1(C_2C_6+2C_1C_2C_6)\left|\dfrac{\lambda_{j-1+r}}{\lambda_j}\right|^k2\rho^k\\
&\leq (2C_2C_6+8C_1(C_2C_6+2C_1C_2C_6))\left|\dfrac{\lambda_{j-1+r}}{\lambda_j}\right|^k\rho^k=C_{11}\left|\dfrac{\lambda_{j-1+r}}{\lambda_j}\right|^k\rho^k.
\end{align*}

\subsubsection{Bound of \(|\varphi(j-1+r)-b(r)|\)}

Define \(C_{12}:=4C_1^2C_6\).

\(\forall r\in[n-j+1]\), direct calculation yields
\begin{align*}
&|\varphi(j-1+r)-b(r)|\\
&=|((I_{n-j+1}+SS^*)^{-1}b)_r-b(r)|\\
&=|((I_{n-j+1}-SS^*(I_{n-j+1}+SS^*)^{-1})b)_r-b(r)|\\
&=|(SS^*(I_{n-j+1}+SS^*)^{-1}b)_r|\\
&\leq\sum_{r'\in[n-j+1]}|(SS^*)(r,r')|\|(I_{n-j+1}+SS^*)^{-1}\|_{\infty}\|b\|_{\infty}\\
&\leq4C_1^2C_6\rho^{2k}\left|\dfrac{\lambda_{j-1+r}}{\lambda_j}\right|^k\sum_{r'\in[n-j+1]}\left|\dfrac{\lambda_{j-1+r'}}{\lambda_j}\right|^k\\
&\leq4C_1^2C_6\rho^k\left|\dfrac{\lambda_{j-1+r}}{\lambda_j}\right|^k=C_{12}\rho^k\left|\dfrac{\lambda_{j-1+r}}{\lambda_j}\right|^k.
\end{align*}

\subsubsection{Bound of \(|v(j-1+r)-\varphi(j-1+r)|\)}

Define \(C_{13}:=C_{11}+C_{12}\).

\(\forall r\in[n-j+1]\), we have
\begin{align*}
|v(j-1+r)-\varphi(j-1+r)|&\leq |v(j-1+r)-b(r)|+|b(r)-\varphi(j-1+r)|\\
&\leq |v(j-1+r)-b(r)|+|\varphi(j-1+r)-b(r)|\\
&\leq C_{11}\rho^k\left|\dfrac{\lambda_{j-1+r}}{\lambda_j}\right|^k+C_{12}\rho^k\left|\dfrac{\lambda_{j-1+r}}{\lambda_j}\right|^k\\
&=C_{13}\rho^k\left|\dfrac{\lambda_{j-1+r}}{\lambda_j}\right|^k.
\end{align*}

\subsection{Bound of \(\|v\|_2\) and \(\|\varphi\|_2\)}

\subsubsection{Bound of \(\|v\|_2\)}

\(\forall i\in[j-1]\), we have
\begin{align*}
|v(i)|&\leq |v(i)+(S'^*W_{22}b)_i|+|(S'^*W_{22}b)_i|\\
&\leq C_8\rho^k\left|\dfrac{\lambda_j}{\lambda_i}\right|^k+C_2C_6\sum_{s\in[n-j+1]}|S'(s,i)|\\
&\leq C_8\rho^k\left|\dfrac{\lambda_j}{\lambda_i}\right|^k+4C_1C_2C_6\left|\dfrac{\lambda_j}{\lambda_i}\right|^k\\
&\leq (C_8+4C_1C_2C_6)\left|\dfrac{\lambda_j}{\lambda_i}\right|^k.
\end{align*}

\(\forall r\in[n-j+1]\), we have
\begin{align*}
|v(j-1+r)|&\leq |v(j-1+r)-b(r)|+|b(r)|\\
&\leq C_{11}\left|\dfrac{\lambda_{j-1+r}}{\lambda_j}\right|^k\rho^k+C_6\left|\dfrac{\lambda_{j-1+r}}{\lambda_j}\right|^k\\
&\leq (C_{11}+C_6)\left|\dfrac{\lambda_{j-1+r}}{\lambda_j}\right|^k.
\end{align*}
And
\[
|v(j)|\geq |b(1)|-|v(j)-b(1)|\geq 1-C_{11}\rho^k.
\]
There exists \(k_4\geq k_3\) such that for \(k\geq k_4\), \(C_{11}\rho^k\leq 1/2\). Suppose \(k\geq k_4\) from now on.

Define \(C_{14}:=(C_8+4C_1C_2C_6)^2+2C_{11}+C_{11}^2+(C_{11}+C_6)^2\).

Therefore we have
\[
\|v\|_2\geq|v(j)|\geq 1-C_{11}\rho^k\geq 1/2,
\]
and
\begin{align*}
&\|v\|_2^2\\
&=\sum_{i\in[j-1]}|v(i)|^2+|v(j)|^2+\sum_{r\in[n-j+1]\setminus\{1\}}|v(j-1+r)|^2\\
&\leq\sum_{i\in[j-1]}(C_8+4C_1C_2C_6)^2\left|\dfrac{\lambda_j}{\lambda_i}\right|^{2k}+(|v(j)-b(1)|+1)^2+\sum_{r\in[n-j+1]\setminus\{1\}}(C_{11}+C_6)^2\left|\dfrac{\lambda_{j-1+r}}{\lambda_j}\right|^{2k}\\
&\leq(C_8+4C_1C_2C_6)^2\rho^k+(C_{11}\rho^k+1)^2+(C_{11}+C_6)^2\rho^k\\
&\leq(C_8+4C_1C_2C_6)^2\rho^k+1+2C_{11}\rho^k+C_{11}^2\rho^k+(C_{11}+C_6)^2\rho^k\\
&\leq 1+((C_8+4C_1C_2C_6)^2+2C_{11}+C_{11}^2+(C_{11}+C_6)^2)\rho^k=1+C_{14}\rho^k.
\end{align*}
In conclusion we have
\[
1/2\leq 1-C_{11}\rho^k\leq\|v\|_2\leq1+C_{14}\rho^k.
\]

\subsubsection{Bound of \(\|\varphi\|_2\)}

\(\forall i\in[j-1]\), we have
\begin{align*}
|\varphi(i)|&\leq|\varphi(i)+(S^*b)_i|+|(S^*b)_i|\\
&\leq C_7\rho^k\left|\dfrac{\lambda_j}{\lambda_i}\right|^k+C_1C_6\sum_{r\in[n-j+1]}\left|\dfrac{\lambda_{j-1+r}}{\lambda_i}\right|^k\\
&\leq C_7\left|\dfrac{\lambda_j}{\lambda_i}\right|^k+2C_1C_6\left|\dfrac{\lambda_j}{\lambda_i}\right|^k\\
&=(C_7+2C_1C_6)\left|\dfrac{\lambda_j}{\lambda_i}\right|^k.
\end{align*}
\(\forall r\in[n-j+1]\), we have
\begin{align*}
|\varphi(j-1+r)|&\leq |\varphi(j-1+r)-b(r)|+|b(r)|\\
&\leq C_{12}\left|\dfrac{\lambda_{j-1+r}}{\lambda_j}\right|^k+C_6\left|\dfrac{\lambda_{j-1+r}}{\lambda_j}\right|^k\\
&\leq (C_{12}+C_6)\left|\dfrac{\lambda_{j-1+r}}{\lambda_j}\right|^k.
\end{align*}
And
\[
|\varphi(j)|\geq|b(1)|-|\varphi(j)-b(1)|\geq1-C_{12}\rho^k.
\]
There exists \(k_5\geq k_4\) such that for \(k\geq k_5\), \(C_{12}\rho^k\leq 1/2\). Suppose \(k\geq k_5\) from now on.

Define \(C_{15}:=(C_7+2C_1C_6)^2+2C_{12}+C_{12}^2+(C_{12}+C_6)^2\).

Therefore we have
\[
\|\varphi\|_2\geq|\varphi(j)|\geq 1-C_{12}\rho^k\geq 1/2
\]
and
\begin{align*}
&\|\varphi\|_2^2\\
&=\sum_{i\in[j-1]}|\varphi(i)|^2+|\varphi(j)|^2+\sum_{r\in[n-j+1]\setminus\{1\}}|\varphi(j-1+r)|^2\\
&\leq\sum_{i\in[j-1]}(C_7+2C_1C_6)^2\left|\dfrac{\lambda_j}{\lambda_i}\right|^{2k}+(|\varphi(j)-b(1)|+|b(1)|)^2+\sum_{r\in[n-j+1]\setminus\{1\}}(C_{12}+C_6)^2\left|\dfrac{\lambda_{j-1+r}}{\lambda_j}\right|^{2k}\\
&\leq(C_7+2C_1C_6)^2\rho^k+(C_{12}\rho^k+1)^2+(C_{12}+C_6)^2\rho^k\\
&\leq(C_7+2C_1C_6)^2\rho^k+1+2C_{12}\rho^k+C_{12}^2\rho^{k}+(C_{12}+C_6)^2\rho^k\\
&\leq 1+((C_7+2C_1C_6)^2+2C_{12}+C_{12}^2+(C_{12}+C_6)^2)\rho^k=1+C_{15}\rho^k.
\end{align*}
In conclusion we have
\[
1/2\leq 1-C_{12}\rho^k\leq\|\varphi\|_2\leq 1+C_{15}\rho^k.
\]

\subsection{Conclusion of the \(j\)-th column}

Direct calculation yields
\[
\left|\dfrac{1}{\|v\|_2}-\dfrac{1}{\|\varphi\|_2}\right|=\dfrac{|\|\varphi\|_2-\|v\|_2|}{\|v\|_2\|\varphi\|_2}\leq 4(C_{11}+C_{12}+C_{14}+C_{15})\rho^k.
\]
Define \(C_{16}:=2C_{10}+4(C_7+2C_1C_6)(C_{11}+C_{12}+C_{14}+C_{15})\).

\(\forall i\in[j-1]\),
\begin{align*}
&||Q(WL_k)(i,j)|-|Q(L_k)(i,j)||\\
&=\left|\dfrac{|v(i)|}{\|v\|_2}-\dfrac{|\varphi(i)|}{\|\varphi\|_2}\right|\\
&=\left|\dfrac{|v(i)|}{\|v\|_2}-\dfrac{|\varphi(i)|}{\|v\|_2}+\dfrac{|\varphi(i)|}{\|v\|_2}-\dfrac{|\varphi(i)|}{\|\varphi\|_2}\right|\\
&\leq\dfrac{||v(i)|-|\varphi(i)||}{\|v\|_2}+|\varphi(i)|\left|\dfrac{1}{\|v\|_2}-\dfrac{1}{\|\varphi\|_2}\right|\\
&\leq 2C_{10}\rho^k\left|\dfrac{\lambda_j}{\lambda_i}\right|^k+(C_7+2C_1C_6)\left|\dfrac{\lambda_j}{\lambda_i}\right|^k4(C_{11}+C_{12}+C_{14}+C_{15})\rho^k\\
&=(2C_{10}+4(C_7+2C_1C_6)(C_{11}+C_{12}+C_{14}+C_{15}))\rho^k\left|\dfrac{\lambda_j}{\lambda_i}\right|^k\\
&=C_{16}\rho^k\left|\dfrac{\lambda_j}{\lambda_i}\right|^k.
\end{align*}
Define \(C_{17}:=2C_{13}+4(C_{12}+C_6)(C_{11}+C_{12}+C_{14}+C_{15})\).

\(\forall r\in[n-j+1]\),
\begin{align*}
&||Q(WL_k)(j-1+r,j)|-|Q(L_k)(j-1+r,j)||\\
&\leq\dfrac{1}{\|v\|_2}||v(j-1+r)|-|\varphi(j-1+r)||+|\varphi(j-1+r)|\left|\dfrac{1}{\|v\|_2}-\dfrac{1}{\|\varphi\|_2}\right|\\
&\leq 2C_{13}\rho^k\left|\dfrac{\lambda_{j-1+r}}{\lambda_j}\right|^k+(C_{12}+C_6)\left|\dfrac{\lambda_{j-1+r}}{\lambda_j}\right|^k4(C_{11}+C_{12}+C_{14}+C_{15})\rho^k\\
&=(2C_{13}+4(C_{12}+C_6)(C_{11}+C_{12}+C_{14}+C_{15}))\rho^k\left|\dfrac{\lambda_{j-1+r}}{\lambda_j}\right|^k\\
&=C_{17}\rho^k\left|\dfrac{\lambda_{j-1+r}}{\lambda_j}\right|^k.
\end{align*}
Define \(C_{18}:=\max\{C_{16},C_{17}\}\). Then \(\forall i\in[j-1]\), we have
\[
||Q(WL_k)(i,j)|-|Q(L_k)(i,j)||\leq C_{16}\rho^k\left|\dfrac{\lambda_j}{\lambda_i}\right|^k\leq C_{18}\rho^k\left|\dfrac{\lambda_j}{\lambda_i}\right|^k.
\]
\(\forall r\in[n-j+1]\), we have
\[
||Q(WL_k)(j-1+r,j)|-|Q(L_k)(j-1+r,j)||\leq C_{17}\rho^k\left|\dfrac{\lambda_{j-1+r}}{\lambda_j}\right|^k\leq C_{18}\rho^k\left|\dfrac{\lambda_{j-1+r}}{\lambda_j}\right|^k.
\]

In conclusion, \(\forall i\in[n]\),
\[
||Q(WL_k)(i,j)|-|Q(L_k)(i,j)||\leq C_{18}\rho^k\left|\dfrac{\lambda_{\max\{i,j\}}}{\lambda_{\min\{i,j\}}}\right|^k.
\]

\subsection{Conclusion of the first column}

Define
\[
\varphi:=L_k(:,1),\qquad v:=(WL_k)(:,1).
\]
Then we have
\[
Q(L_k)(:,1)=\varphi/\|\varphi\|_2,\qquad Q(WL_k)(:,1)=v/\|v\|_2.
\]
\(\forall i\in[n]\), we have
\[
\varphi(i)=L(i,1)\left(\dfrac{\lambda_i}{\lambda_1}\right)^k
\]
and
\[
v(i)=\sum_{r\in[n]\setminus[i-1]}W(i,r)L(r,1)\left(\dfrac{\lambda_r}{\lambda_1}\right)^k.
\]
Therefore \(\forall i\in[n]\),
\begin{align*}
|v(i)-\varphi(i)|&=\left|\sum_{r\in[n]\setminus[i-1]}W(i,r)L(r,1)\left(\dfrac{\lambda_r}{\lambda_1}\right)^k-L(i,1)\left(\dfrac{\lambda_i}{\lambda_1}\right)^k\right|\\
&=\left|\sum_{r\in[n]\setminus[i]}W(i,r)L(r,1)\left(\dfrac{\lambda_r}{\lambda_1}\right)^k\right|\\
&\leq C_2C_6\left|\dfrac{\lambda_i}{\lambda_1}\right|^k\sum_{r\in[n]\setminus[i]}\left|\dfrac{\lambda_r}{\lambda_i}\right|^k\\
&\leq2C_2C_6\left|\dfrac{\lambda_i}{\lambda_1}\right|^k\rho^k.
\end{align*}
Notice that
\[
\|\varphi\|_2\geq|\varphi(1)|=1
\]
and
\[
\|\varphi\|_2^2=\sum_{i\in[n]}|\varphi(i)|^2=\sum_{i\in[n]}|L(i,1)|^2\left|\dfrac{\lambda_i}{\lambda_1}\right|^{2k}\leq1+C_6^2\sum_{i\in[n]\setminus\{1\}}\left|\dfrac{\lambda_i}{\lambda_1}\right|^{2k}\leq1+C_6^2\rho^k.
\]
We have
\[
1\leq\|\varphi\|_2\leq1+C_6^2\rho^k.
\]
There exists \(k_0\geq k_5\) such that for \(k\geq k_0\), \(2C_2C_6\rho^k\leq 1/2\). Suppose \(k\geq k_0\) from now on.
\begin{align*}
\|v\|_2&\geq|v(1)|=\left|\sum_{r\in[n]}W(1,r)L(r,1)\left(\dfrac{\lambda_r}{\lambda_1}\right)^k\right|\\
&\geq1-\sum_{r\in[n]\setminus\{1\}}|W(1,r)||L(r,1)|\left|\dfrac{\lambda_r}{\lambda_1}\right|^k\\
&\geq1-2C_2C_6\rho^k\\
&\geq 1/2.
\end{align*}
Define \(C_{19}:=4C_2C_6+6C_2^2C_6^2\).

We have
\begin{align*}
\|v\|_2^2&=|v(1)|^2+\sum_{i\in[n]\setminus\{1\}}|v(i)|^2\\
&=\left|\sum_{r\in[n]\setminus[0]}W(1,r)L(r,1)\left(\dfrac{\lambda_r}{\lambda_1}\right)^k\right|^2+\sum_{i\in[n]\setminus\{1\}}\left|\sum_{r\in[n]\setminus[i-1]}W(i,r)L(r,1)\left(\dfrac{\lambda_r}{\lambda_1}\right)^k\right|^2\\
&\leq\left(1+C_2C_6\sum_{r\in[n]\setminus\{1\}}\left|\dfrac{\lambda_r}{\lambda_1}\right|^k\right)^2+\sum_{i\in[n]\setminus\{1\}}\left(C_2C_6\sum_{r\in[n]\setminus[i-1]}\left|\dfrac{\lambda_r}{\lambda_1}\right|^k\right)^2\\
&\leq\left(1+2C_2C_6\rho^k\right)^2+4C_2^2C_6^2\rho^k\\
&\leq1+4C_2C_6\rho^k+6C_2^2C_6^2\rho^k\\
&\leq1+(4C_2C_6+6C_2^2C_6^2)\rho^k=1+C_{19}\rho^k.
\end{align*}
In conclusion, we have
\[
1/2\leq 1-2C_2C_6\rho^k\leq\|v\|_2\leq 1+C_{19}\rho^k.
\]
Direct calculation yields
\[
\left|\dfrac{1}{\|v\|_2}-\dfrac{1}{\|\varphi\|_2}\right|=\dfrac{|\|\varphi\|_2-\|v\|_2|}{\|v\|_2\|\varphi\|_2}\leq2(C_6^2+2C_2C_6+C_{19})\rho^k.
\]
Define \(C_{20}:=4C_2C_6+2C_6(C_6^2+2C_2C_6+C_{19})\).

\(\forall i\in[n]\), we have
\begin{align*}
&||Q(WL_k)(i,1)|-|Q(L_k)(i,1)||\\
&\leq\dfrac{1}{\|v\|_2}|v(i)-\varphi(i)|+|\varphi(i)|\left|\dfrac{1}{\|v\|_2}-\dfrac{1}{\|\varphi\|_2}\right|\\
&\leq4C_2C_6\rho^k\left|\dfrac{\lambda_i}{\lambda_1}\right|^k+2C_6(C_6^2+2C_2C_6+C_{19})\rho^k\left|\dfrac{\lambda_i}{\lambda_1}\right|^k\\
&=(4C_2C_6+2C_6(C_6^2+2C_2C_6+C_{19}))\rho^k\left|\dfrac{\lambda_i}{\lambda_1}\right|^k\\
&=C_{20}\rho^k\left|\dfrac{\lambda_i}{\lambda_1}\right|^k=C_{20}\rho^k\left|\dfrac{\lambda_{\max\{i,1\}}}{\lambda_{\min\{i,1\}}}\right|^k.
\end{align*}

\subsection{Finish the proof of Lemma \ref{lemma:key lemma}}

Define \(C_{0}:=\max\{C_{18},C_{20}\}\). \(\forall(i,j)\in[n]\times[m]\),
\[
||Q(WL_k)(i,j)|-|Q(L_k)(i,j)||\leq C_0\rho^k\left|\dfrac{\lambda_{\max\{i,j\}}}{\lambda_{\min\{i,j\}}}\right|^k.
\]

\bibliographystyle{plain}
\bibliography{references}

@book {saad_book,
    AUTHOR = {Saad, Yousef},
     TITLE = {Numerical methods for large eigenvalue problems},
    SERIES = {Classics in Applied Mathematics},
    VOLUME = {66},
   EDITION = {Revised},
 PUBLISHER = {Society for Industrial and Applied Mathematics (SIAM),
              Philadelphia, PA},
      YEAR = {2011},
     PAGES = {xvi+276},
      ISBN = {978-1-611970-72-2},
   MRCLASS = {65F15 (01A75 65-02 65F50)},
  MRNUMBER = {3396212},
       DOI = {10.1137/1.9781611970739.ch1},
       URL = {https://doi.org/10.1137/1.9781611970739.ch1},
}

@book {golub_book,
    AUTHOR = {Golub, Gene H. and Van Loan, Charles F.},
     TITLE = {Matrix computations},
    SERIES = {Johns Hopkins Studies in the Mathematical Sciences},
   EDITION = {Fourth},
 PUBLISHER = {Johns Hopkins University Press, Baltimore, MD},
      YEAR = {2013},
     PAGES = {xiv+756},
      ISBN = {978-1-4214-0794-4; 1-4214-0794-9; 978-1-4214-0859-0},
   MRCLASS = {65-02 (65Fxx)},
  MRNUMBER = {3024913},
MRREVIEWER = {J\"org\ Liesen},
}

@book {parlett_book,
    AUTHOR = {Parlett, Beresford N.},
     TITLE = {The symmetric eigenvalue problem},
    SERIES = {Classics in Applied Mathematics},
    VOLUME = {20},
      NOTE = {Corrected reprint of the 1980 original},
 PUBLISHER = {Society for Industrial and Applied Mathematics (SIAM),
              Philadelphia, PA},
      YEAR = {1998},
     PAGES = {xxiv+398},
      ISBN = {0-89871-402-8},
   MRCLASS = {65F15 (15A18)},
  MRNUMBER = {1490034},
MRREVIEWER = {F.\ Szidarovszky},
       DOI = {10.1137/1.9781611971163},
       URL = {https://doi.org/10.1137/1.9781611971163},
}

@inproceedings{NEURIPS2020_3d8e03e8,
 author = {Charisopoulos, Vasileios and Benson, Austin R and Damle, Anil},
 booktitle = {Advances in Neural Information Processing Systems},
 editor = {H. Larochelle and M. Ranzato and R. Hadsell and M.F. Balcan and H. Lin},
 pages = {5644--5655},
 publisher = {Curran Associates, Inc.},
 title = {Entrywise convergence of iterative methods for eigenproblems},
 url = {https://proceedings.neurips.cc/paper_files/paper/2020/file/3d8e03e8b133b16f13a586f0c01b6866-Paper.pdf},
 volume = {33},
 year = {2020}
}

\end{document}